\documentclass[12pt]{article}
\usepackage{amsmath}
\usepackage{amssymb}
\usepackage{amsthm}

\newtheorem{theorem}{Theorem}[section]
\newtheorem{lemma}[theorem]{Lemma}
\newtheorem{corollary}[theorem]{Corollary}

\theoremstyle{definition}
\newtheorem{definition}[theorem]{Definition}
\newtheorem{example}[theorem]{Example}

\theoremstyle{remark}
\newtheorem*{remark}{Remark}

\numberwithin{equation}{section}

\def \BV {{\rm BV}}
\def \Chi {\raise .3ex \hbox{\large $\chi$}}

\def\R{\mathbb{R}}
\def\Z{\mathbb{Z}} 
\def\N{\mathbb{N}}

\def\argmin{\mathop{\rm argmin}}
\def\Div{{\rm div}}
\def\dist{{\rm dist}}
\def\diam{{\rm diam}}

\def\e{\epsilon}

\title{Multiscale Decompositions and Optimization}

\author{Xiaohui Wang}
\begin{document}

\maketitle

\begin{abstract}
In this thesis,  the following type Tikhonov regularization problem will be systematically studied:
\[(u_t,v_t):=\argmin_{u+v=f}~\{\|v\|_X+t\|u\|_Y\},\]
where $Y$ is a smooth space such as  a $\BV$ space or a Sobolev space and $X$ is the space in which we measure distortion.
Examples of the above problem occur in denoising in image processing,  in numerically treating inverse problems,  and
in the sparse recovery problem of compressed sensing. It is also at the heart of interpolation of linear operators by the real method of interpolation.
We shall characterize the minimizing pair $(u_t,v_t)$ for $(X,Y):=(L_2(\Omega),\BV(\Omega))$ as a primary example and generalize Yves Meyer's result in \cite{Meyer} and Antonin Chambolle's result in \cite{Chambolle}. After that, the following multiscale decomposition scheme will be studied:
\[u_{k+1}:=\argmin_{u\in \BV(\Omega)\cap L_2(\Omega)}~\{\frac{1}{2}\|f-u\|^2_{L_2}+t_{k}|u-u_k|_{\BV}\},\]
where $u_0=0$ and $\Omega$ is a bounded Lipschitz domain in $\R^d$. This method was introduced by Eitan Tadmor et al. and we shall improve the $L_2$ convergence result in \cite{Tadmor}. Other pairs such as $(X,Y):=(L_p,W^{1}(L_\tau))$ and $(X,Y):=(\ell_2,\ell_p)$ will also be mentioned. In the end, the numerical implementation for $(X,Y):=(L_2(\Omega),\BV(\Omega))$ and the corresponding convergence results will be given.

\end{abstract}


\section{Introduction: The Importance of Research}
Many problems in optimization and applied mathematics center on decomposing a given function $f$
into a sum of two functions with prescribed properties. Typically, one of these functions is called a {\em good} function $u$ and represents the properties of $f$ we wish to maintain while the second part $v$ represents  error/distortion or noise in the stochastic setting. Examples occur in denoising in image processing,  in numerically treating inverse problems,  and
in the sparse recovery problem of compressed sensing. The general problem of decomposing a function
as a sum of two functions is also at the heart of interpolation of linear operators by the real method of interpolation. My research explores the mathematics behind such decompositions and their numerical implementation.

One can formulate the decomposition problem for any pair of Banach spaces $X$ and $Y$ with $Y$ the space of good functions and $X$ the space in which we measure distortion. Given a real number $t>0$, we consider the minimization problem:
\begin{equation}
\label{Kfunctional}
K(f,t):= \inf_{f=u+v} ~\{\|v\|_X+t\|u\|_Y\}.
\end{equation}
$K(f,t)$ is called the K-functional for the pair $(X,Y)$. The pair $(u_t,v_t)$ which minimizes $K(f,t)$ is the  Tikhonov regularization pair:
\begin{equation}
\label{Tikhonov}
(u_t,v_t):=\argmin_{u+v=f}~\{\|v\|_X+t\|u\|_Y\}.
\end{equation}
One can usually prove (by compactness argument and strict convexity) that there exists a unique solution $(u_t,v_t)$ for problem (\ref{Tikhonov}). As we vary $t$, we obtain different decompositions. These decompositions describe how $f$ sits relative to $X$ and $Y$.
There are many variants of (\ref{Tikhonov}) that are commonly used.  For example, the norm of $Y$ can be replaced by a semi-norm or a quasi-norm and sometimes the norm with respect to $X$ is raised to a power.

While the above formulation can be defined for any pair $(X,Y)$ and any $f$ in $X+Y$,  in applications we are interested in specific pairs. One common setting, and the first  one of interest to me, is when $X=L_p$ and $Y$ is a smooth space such as  a Sobolev space or a BV space. This particular case appears in many problems of image processing, optimization, compression, and encoding. We shall study various questions associated to such decompositions.

The main problems to be investigated in this thesis are:

\begin{description}
\item[(i)~~] Given $f$ and $t>0$, characterize the minimizing pair $(u_t,v_t)$.
\item[(ii)~] Find  an analytic expression for $K(f,t)$ in terms of classical quantities and thereby characterize the interpolation spaces for a given pair $(X,Y)$.
\item[(iii)] Multiscale decompositions corresponding to the pair $(X,Y)$ that can be derived from the characterization of the minimizing pair.
\item[(iv)~] Numerical methods for  computing this decomposition or something close to it.
\end{description}

The structure of this thesis is as following:
\begin{description}
\item[Chapter 1:]Introduction: The Importance of Research.
\item[Chapter 2:]Basic properties of $\BV(\Omega)$ and Hausdorff measure.
\item[Chapter 3:]Decomposition for the pair $(L_2(\Omega),\BV(\Omega))$, where $\Omega$ is a bounded Lipschitz domain in $\R^d$. In this section, we first characterize the minimizing pair $(u_t,v_t)$ by studying the Euler-Lagrange equation associated to (\ref{Tikhonov}) which includes giving an appropriate setting for the boundary condition. We generalize Yves Meyer's result in \cite{Meyer} and Antonin Chambolle's result in \cite{Chambolle} on the properties of $(u_t,v_t)$. Then the expression of $K(f,t)$ follows as a simple consequence. In addition, we propose simpler proofs about characterizing the subdifferential of $\BV$ semi-norm which were first proved in \cite{Andreu}.
\item[Chapter 4:]Multiscale decompositions corresponding to the pair $(L_2,\BV)$. In this section, we study the scheme introduced by Eitan Tadmor et al. under the general framework of {\em Inverse Scale Space Methods} and improve the $L_2$ convergence result in \cite{Tadmor}.
\item[Chapter 5:]Decomposition for $(X,Y):=(L_p,W^{1}(L_\tau))$ with $1/\tau:=1/p+1/d$.
\item[Chapter 6:]Decomposition for $(X,Y):=(\ell_2,\ell_p)$ with $1\leq p<\infty$.
\item[Chapter 7:]Numerical implementation for $(X,Y):=(L_2(\Omega),\BV(\Omega))$ and the corresponding convergence results.
\end{description}

\section{The Space $\BV(\Omega)$ and Hausdorff Measure}
In this section, we will introduce some basic facts about the space $\BV(\Omega)$, where $\Omega$ is a bounded Lipschitz domain.
First, we need to give a definition of the Lipschitz domain.
\begin{definition}[Lipschitz Domain]
An open set $\Omega\subset \R^d$ is a {\em Lipschitz domain} if for any $x_0\in\partial \Omega$ there exists $r>0$ and a Lipschitz function $\Phi:\R^{d-1}\rightarrow \R$ such that~--~upon relabeling and reorienting the coordinates axis~--~we have
\[\Omega\cap B(x_0,r)=\{x\in B(x_0,r):x_d>\Phi(x_1,\dotsc,x_{d-1})\}.\]
\end{definition}
In the following text, without specifically mentioned, we will assume $\Omega$ is a bounded Lipschitz domain in $\R^d$ and use the following notations:
\begin{itemize}
\item $|\Omega|$: Lebesgue measure of $\Omega$.
\item $x:=(x_1,x_2,\dotsc,x_d)$.
\item $|x|:=\sqrt{\sum_{i=1}^{d}x^2_i}$.
\item $\nabla u:=(\frac{\partial u}{\partial {x_1}},\frac{\partial u}{\partial {x_2}},\dotsc,\frac{\partial u}{\partial {x_d}})$.
\item $D^{\alpha}u:=\partial^{\alpha_1}_{x_1} \partial^{\alpha_2}_{x_2} \dotsm \partial^{\alpha_d}_{x_d}u$, where $\alpha:=(\alpha_1,\alpha_2,\dotsc,\alpha_d)$ and $\alpha_i\in \N\cup \{0\}$ for $1\leq i \leq d$.
\end{itemize}

Before introducing the space $\BV(\Omega)$, we need to give definitions of {\em weak derivative} and {\em measure}.

\begin{definition}[Weak Derivative]
Let $u\in L_1(\Omega)$. For a given multi-index $\alpha$, a function $v\in L_1(\Omega)$ is called the $\alpha^{\rm th}$ weak derivative of $u$ if
\[\int_{\Omega}v\phi~dx=(-1)^{|\alpha|}\int_{\Omega}uD^{\alpha}\phi~dx\]
for all $\phi\in C^{\infty}_{0}(\Omega)$. Here $|\alpha|:=\alpha_1+\alpha_2+\dotsb+\alpha_d$.
\end{definition}

\begin{definition}[Measure]
The $\alpha^{\rm th}$ weak derivative of $u$ is called a measure if there exists a regular Borel (signed) measure $\mu$ on $\Omega$ such that
\[\int_{\Omega}\phi ~d\mu=(-1)^{|\alpha|}\int_{\Omega}u D^{\alpha}\phi~dx\]
for all $\phi \in C^{\infty}_{0}(\Omega)$. In addition $|D^{\alpha}u|(\Omega)$ denotes the total variation of the measure $\mu$.
\end{definition}

\begin{definition}
\[\BV(\Omega):=L_1(\Omega)\cap\{u:D^{\alpha}u\mbox{~is~a~measure,~}|D^{\alpha}u|(\Omega)<\infty,|\alpha|=1\}.\]
In addition, the BV semi-norm $|u|_{\BV}$ can be defined as:
\[|u|_{\BV}=\int_{\Omega}|Du|:=\sup~\{\int_{\Omega}u\Div(\phi)~dx:\phi\in C^{\infty}_{0}(\Omega;\R^d),|\phi(x)|\leq 1{\rm~for~}x\in \Omega\}.\]
\end{definition}
\begin{remark}It is easy to see, for $u \in W^{1}(L_1(\Omega))$, $|u|_{\BV}=\int_{\Omega}|\nabla u|~dx$.\end{remark}
\begin{theorem}[Coarea Formula]
Let $u\in \BV(\Omega)$ and define $E_{t}:=\{x\in \Omega:u(x)<t\}$. Then
\[\int_{\Omega}|Du|=\int^{\infty}_{-\infty}dt\int_{\Omega}|D\Chi_{E_t}|.\]
\end{theorem}
\begin{proof}
See Theorem 1.23 in \cite{Giusti}.
\end{proof}
\begin{definition}[Hausdorff Measure]
For set $E\subset \R^{d}$, $0\leq k<\infty$ and $0<\delta\leq \infty$, we define
\[H^{\delta}_{k}(E):=\omega_{k}2^{-k}\inf~\{\sum^{\infty}_{j=1}(\diam~S_j)^k:E\subset \mathop{\bigcup}^{\infty}_{j=1}S_{j},\diam~S_{j}<\delta\}\]
and
\[H^{k}(E):=\lim_{\delta\rightarrow 0}H^{\delta}_{k}(E)=\sup_{\delta}H^{\delta}_{k}(E),\]
where $\omega_{k}=\Gamma(\frac{1}{2})^{k}/\Gamma(\frac{k}{2}+1)$, $k\geq 0$, is the Lebesgue measure of the unit ball in $\R^{k}$.\\
$H^k$ is called the $k$-dimensional Hausdorff measure.
\end{definition}
\begin{example}
Suppose $E\subset \Omega$ has $C^{2}$ boundary and consider $\Chi_{E}$ the characteristic function of $E$, then
\[|\Chi_{E}|_{\BV}=H^{d-1}(\partial E\cap \Omega).\]
\end{example}
\begin{proof}
See Example 1.4 in \cite{Giusti}.
\end{proof}

For a bounded Lipschitz domain $\Omega$, the outward unit normal vector $\nu(x):=(\nu_1,\nu_2,...,\nu_d)$ is defined $H^{d-1}$-a.e. on $\partial \Omega$. 
Then we have the following generalized Gauss-Green theorem:
\[\int_\Omega\Div(F)~dx=\int_{\partial \Omega}F(x)\cdot \nu(x)~dH^{d-1}\]
whenever $F\in C^1(\bar{\Omega};\R^d)$. For detailed exposition, please refer to \cite{Ziemer}.

\section{Decomposition for $(X,Y)=(L_2(\Omega),\BV(\Omega))$}
While there are many settings and potential decompositions that we shall discuss,  a particular problem which  is of high interest and is a primary example of the goal of my research is the problem of decomposing a function $f\in L_2(\Omega)$, where $\Omega$ is a bounded Lipschitz domain in ${\R}^d$, into a sum of an $L_2$ function and a $\BV$ function. For any prescribed $t>0$, we define the   pair $(u_t,v_t)$ as the solution of  the following minimization problem:
\begin{equation}
\label{equ:L2-BV-pair}
(u_t,v_t):=\argmin_{u+v=f}~\{\frac{1}{2}\|v\|^2_{L_2}+t |u|_{\BV}\}.
\end{equation}
If we define $T(u):=\frac{1}{2}\|f-u\|^{2}_{L_2}+t|u|_{\BV}$, then problem (\ref{equ:L2-BV-pair}) is equivalent to the following problem:
\begin{equation}
\label{equ:L2-BV}
u_t:=\argmin_{u\in \BV(\Omega)\cap L_2(\Omega)}T(u).
\end{equation}
Problem (\ref{equ:L2-BV}) is closely related to the following constrained minimization problem:
\begin{equation}
\label{constrain-min}
\min_{u\in \BV(\Omega)\cap L_2(\Omega)}~J(u) {\rm~~subject~to~~}\int_\Omega f~dx=\int_\Omega u~dx ~{\rm and}~\|f-u\|_{L_2}=\sigma,
\end{equation}
where $J(u):=|u|_{\BV}$.
It is widely used in  image denoising where it is called Rudin-Osher-Fatemi model for $\Omega\subset \R^2$ (see \cite{Osher1}).    If  $f$ is a  given noisy image, then $u_t$ captures the main features of $f$ and $v_t$ contains the oscillatory patterns of texture or the inherent noise in the image.  In the stochastic setting, a central question is what is the best choice of $t$.

Before introducing our main results, we need to spend a few words on the rigorous definition of the solution of (\ref{equ:L2-BV}).

To derive the Euler-Lagrange equation associated to problem (\ref{equ:L2-BV}), we first consider a special case that $u\in C^1(\bar{\Omega})$ and $\partial \Omega$ is $C^1$. Set $T(u):=\frac{1}{2}\|f-u\|_{L_2}^2+tJ(u)$ with $J(u):=\int_{\Omega}|\nabla u|~dx$. Consider the following minimization problem:
\[\min_{u\in S}T(u),\]
where \[S:=\{u\in C^1(\bar{\Omega})~:~\frac{\partial u}{\partial \nu}=0 {\rm~for~} x \in \partial \Omega {\rm~and~} \nabla u \neq 0 {\rm~for~all~} x \in \Omega\}\]
is the admissible set.

Notice $\nabla(|x|)=\frac{x}{|x|}$ for $x\neq 0$. We can thus calculate the Gateaux derivative of $J(u)$ for $u\in S$. Given $u \in S$, for any $h \in C^1(\bar{\Omega})$ with $\frac{\partial h}{\partial \nu}=0$ on $\partial \Omega$, when $\epsilon$ small enough, we have $u+\epsilon h \in S$. Hence:
\begin{eqnarray*}
\delta J(u;h)&=&\lim_{\epsilon \rightarrow 0}\frac{1}{\epsilon}\{J(u+\epsilon h)-J(u)\}\\
&=&\lim_{\epsilon \rightarrow 0}\frac{1}{\epsilon}\int_{\Omega}\{|\nabla (u+\epsilon h)|-|\nabla u|\}~dx\\
&=&\lim_{\epsilon \rightarrow 0} \frac{1}{\epsilon}\int_{\Omega}\{\nabla(|x|)|_{x=\nabla u}\cdot \epsilon \nabla h+O(\epsilon^2)\}~dx\\
&=&\int_{\Omega} \frac{\nabla u}{|\nabla u|} \cdot \nabla h~dx\\
&=&\int_{\Omega} -\Div(\frac{\nabla u}{|\nabla u|})h~dx+\int_{\Omega}\Div(\frac{\nabla u}{|\nabla u|}h)~dx\\
&=&\int_{\Omega} -\Div(\frac{\nabla u}{|\nabla u|})h~dx+\int_{\partial \Omega}\frac{h}{|\nabla u|}\frac{\partial u}{\partial \nu}~ds\\
&=&\int_{\Omega} -\Div(\frac{\nabla u}{|\nabla u|})h~dx,
\end{eqnarray*}
where $\frac{\partial u}{\partial \nu}:=\nabla u \cdot \nu$.

The reason why we choose Neumann boundary condition is to make $\int_{\Omega}u~dx=\int_{\Omega}f~dx$, which means the error/distortion $v=f-u$ has mean value zero. As we shall show below, $\int_{\Omega}u~dx=\int_{\Omega}f~dx$ will automatically be satisfied when $u$ is a minimizer for problem  (\ref{equ:L2-BV}).

The necessary condition for $u$ to be a minimizer is: $\delta T(u;h)=0$ which means
\[t\delta J(u;h)- \int_\Omega(f-u)h~dx=0\]
for any $h \in C^1(\bar{\Omega})$ with $\frac{\partial h}{\partial \nu}=0$.
Hence, we can {\em informally} write the Euler-Lagrange equation associated to problem (\ref{equ:L2-BV}) as:
\[
\left \{\begin{array}{ll}u-t\Div(\frac{\nabla u}{|\nabla u|})=f&{\rm in}~ \Omega
\\ \frac{\partial u}{\partial \nu}=0&{\rm on}~ \partial \Omega\end{array} \right.
\]

Now we come back to the more general case which $u \in \BV(\Omega)\cap L_2(\Omega)$ with
\[J(u):=|u|_{\BV}=\sup~\{\int_{\Omega}u\Div(\phi)~dx:\phi \in V\},\]
where $V:=\{\phi\in C_{0}^1(\Omega;\R^d):|\phi(x)|\leq 1\mbox{~for~}x\in \Omega\}$.

We can extend the domain of $J(u)$ to $L_2(\Omega)$ in the following way:
\[J(u):=\left \{ \begin{array}{ll}|u|_{\BV}&u\in \BV(\Omega)\cap L_2(\Omega)\\+\infty&u\in L_2(\Omega) \backslash  (\BV(\Omega)\cap L_2(\Omega))\end{array} \right.\]
In this way, we can also define 
\[T(u):=\frac{1}{2}\|f-u\|_{L_2}^2+tJ(u)\]
on $L_2(\Omega)$. In the following text, without specific mention, we will assume $J(u)$ and $T(u)$ defined on the whole space of $L_2(\Omega)$ as above. It is easy to check that $J(u)$ and $T(u)$ defined in this way are proper convex functionals on $L_2(\Omega)$.%
\begin{lemma}
\label{semi-contin}
$J(u)$ is weakly lower semi-continuous with respect to $L_p(1\leq p \leq \infty)$ topology, i.e. if $u_n\rightharpoonup u$ weakly in $L_p$, we have: $J(u)\leq \liminf_{n\rightarrow \infty} J(u_n)$. In addition, if $\liminf_{n\rightarrow \infty} J(u_n)< \infty$, we have $u \in \BV(\Omega)$.
\end{lemma}
\begin{proof}
For any $\phi\in V$, where  $V:=\{\phi\in C_{0}^1(\Omega;\R^d):|\phi(x)|\leq 1\mbox{~for~}x\in \Omega\}$, we have
\begin{eqnarray*}
\int_{\Omega}u\Div(\phi)~dx&=&\lim_{n\rightarrow \infty}\int_{\Omega}u_n\Div(\phi)~dx\\
&=&\liminf_{n\rightarrow \infty}\int_{\Omega}u_n\Div(\phi)~dx\\
&\leq&\liminf_{n\rightarrow \infty}J(u_n).
\end{eqnarray*}
Take $\phi$ over the set $V$, we get $J(u)\leq \liminf_{n\rightarrow \infty}J(u_n)$.
\end{proof}
Now we give the existence and uniqueness result for problem (\ref{equ:L2-BV}) without invoking the associated Euler-Lagrange equation.
\begin{theorem}
\label{min}
For $t > 0$, there exists a unique minimizer $u_t$ for problem (\ref{equ:L2-BV}). In addition, $u_t$ solves (\ref{constrain-min}) for $\sigma=\|f-u_t\|_{L_2} \leq \|f-\bar{f}\|_{L_2}$ and $\int_{\Omega}u_t~dx=\int_{\Omega}f~dx$, where $\bar{f}=\frac{1}{|\Omega|}\int_{\Omega}f~dx$.
\end{theorem}
\begin{proof}
Let ${u_n}$ be a minimizing sequence for $T(u)$. Since $\|u\|_{L_2}\leq \|f\|_{L_2}+\sqrt{2T(u)}$, $\{\|u_n\|_{L_2}\}$ are bounded. Since $L_p$ is reflexive when $1<p<\infty$, there exists a subsequence $\{u_{n_j}\}$ such that $u_{n_j}\rightharpoonup u_t$ weakly in $L_2$. By the weakly lower semi-continuity of $J(u)$ as in lemma \ref{semi-contin}, we have:
\[tJ(u_t)+\frac{1}{2}\|f-u_t\|^2_{L_2}\leq \liminf_{j\rightarrow\infty}~\{tJ(u_{n_j})+\frac{1}{2}\|f-u_{n_j}\|^2_{L_2}\}=\min_{u\in \BV(\Omega)\cap L_2(\Omega)}T(u).\]
Hence $u_t$ solves (\ref{equ:L2-BV}) and the uniqueness of the minimizer follows immediately from strict convexity of $T(u)$.

Suppose $\sigma=\|f-u_t\|_{L_2} > \|f-\bar{f}\|_{L_2}$, then
\[tJ(\bar{f})+\frac{1}{2}\|f-\bar{f}\|^2_{L_2}<\frac{1}{2}\|f-u_t\|^2_{L_2} \leq tJ(u_t)+\frac{1}{2}\|f-u_t\|^2_{L_2},\]
which is contradictory with the definition of the minimizer. So we have  $\sigma=\|f-u_t\|_{L_2} \leq \|f-\bar{f}\|_{L_2}$.

Similarly, if $\int_{\Omega}u_t~dx\neq \int_{\Omega}f~dx$, let $c=\frac{1}{|\Omega|}\int_{\Omega}(f-u_t)~dx$. Then we have $\|f-(u_t+c)\|_{L_2}<\|f-u_t\|_{L_2}$ while $J(u_t+c)=J(u_t)$, which is contradictory with the definition of the minimizer.
\end{proof}
\begin{remark}There is an alternative way to get the existence proof by compactness argument for $\BV(\Omega)$ (See \cite{Acar}).
\end{remark}

To characterize the minimizer of problem (\ref{equ:L2-BV}), we have to come back to the PDE approach. The associated Euler-Lagrange equation can be formally written as:
\begin{equation}
\label{Euler}
\left \{\begin{array}{ll}u-t\Div(\frac{D u}{|D u|})=f&\mbox{in}~ \Omega
\\ \frac{\partial u}{\partial \nu}=0&{\rm on}~ \partial \Omega\end{array} \right.
\end{equation}
To understand this equation correctly, we first need to give a rigorous definition of the boundary condition and the nonlinear operator $-\Div(\frac{D u}{|D u|})$.

\subsection{Definition of Neumann Boundary Condition}
Throughout this section we frequently make use of results shown by Anzellotti in \cite{Anzellotti}.
To define the Neumann boundary condition in the sense of trace, we shall consider the following spaces:
\begin{itemize}
\item $\BV(\Omega)_p:=\BV(\Omega)\cap L_p(\Omega)$.
\item $X(\Omega)_q=\{z \in L_{\infty}(\Omega;\R^d):\Div(z)\in L_q(\Omega)\}$, where $\Div(z)$ is defined in the sense of distribution: $\int_\Omega \Div(z) \phi~dx=-\int_\Omega z\cdot \nabla \phi~dx$ for any $\phi \in C^\infty_0(\Omega)$. Here $\|z\|_{L_{\infty}}:=\|\sqrt{\sum_{i=1}^{d}z^2_i}\|_{L_\infty}$.
\end{itemize}
Set
\[q:=\left \{\begin{array}{ll}\infty & p=1\\ \frac{p}{p-1} & 1<p<\infty\\1 & p=\infty\end{array}\right.\]
If $z \in X(\Omega)_{q}$ and $w \in \BV(\Omega)_p$, we can define the functional $(z,Dw):C^{\infty}_{0}(\Omega)\mapsto \R$ by the formula:
\[\langle (z,Dw), \phi \rangle:=-\int_{\Omega}w\phi \Div(z)~dx-\int_{\Omega}wz\cdot\nabla{\phi}~dx.\]
\begin{theorem}
\label{measure}
The functional $(z,Dw)$ defined above for $z \in X(\Omega)_{q}$ and $w \in \BV(\Omega)_p$ is a Radon measure on $\Omega$, and
\[\int_{\Omega}(z,Dw)=\int_{\Omega}z\cdot\nabla w~dx\]
for $w\in W^{1}(L_1(\Omega))\cap L_{p}(\Omega)$. In addition, we have 
$|\int_B (z,Dw)|\leq \int_B |(z,Dw)|\leq \|z\|_{L_{\infty}}\int_{B} |Dw|$ for any Borel set $B\subset \Omega$.
\end{theorem}
\begin{proof}
Given $w\in \BV(\Omega)_{p}$, we can find a sequence $\{w_{n}\}\subset C^{\infty}(\Omega)\cap \BV(\Omega)_p$(see \cite{Giusti}) such that:
\[w_{n}\rightarrow w~\mbox{in~}L_{p}(\Omega)~\mbox{and}~\limsup_{n\rightarrow \infty}\int_{\bar{A}\cap \Omega}|Dw_{n}|\leq \int_{\bar{A}\cap \Omega}|Dw|~{\rm for~any~open~set~}A \subset \Omega.\]
Take any $\phi \in C^{\infty}_{0}(A)$ and consider an open set $V$ such that ${\rm supp}(\phi)\subset V \subset\subset A$, then we have
\[\langle (z,Dw_{n}), \phi \rangle=-\int_{V}w_n\Div(\phi z)~dx=\int_{V}\phi z\cdot \nabla w_{n}~dx.\]
So
\[|\langle (z,Dw_{n}), \phi \rangle|\leq \|\phi\|_{L_{\infty}(V)}\|z\|_{L_{\infty}(V)}\int_{V}|\nabla w_{n}|~dx.\]
Taking the limit for $n\rightarrow \infty$, we get
\[|\langle (z,Dw), \phi \rangle|\leq \|\phi\|_{L_{\infty}(V)}\|z\|_{L_{\infty}(V)}\int_{\bar{V}\cap \Omega}|D w|\leq \|\phi\|_{L_\infty}\|z\|_{L_\infty}\int_{A}|Dw|~~\mbox{for any }\phi \in C^{\infty}_{0}(A).\]
So $(z,Dw)$ is a Radon measure and we have 
$|\int_B (z,Dw)|\leq \int_B |(z,Dw)|\leq \|z\|_{L_{\infty}}\int_{B} |Dw|$ for any Borel set $B\subset \Omega$.

Since $\langle (z,Dw), \phi \rangle=-\int_{\Omega}w\Div(\phi z)~dx=\int_{\Omega}\phi z\cdot \nabla w~dx$ for any $w\in W^{1}(L_1(\Omega))\cap L_p(\Omega)$ and $\phi\in C^{\infty}_{0}(\Omega)$,
we get $\int_{\Omega}(z,Dw)=\int_{\Omega}z\cdot \nabla w~dx$ for $w\in W^{1}(L_1(\Omega))\cap L_{p}(\Omega)$.
\end{proof}
\begin{theorem}
\label{dense}
Let $z \in X(\Omega)_{q}$, $w \in \BV(\Omega)_p$ and $(z,Dw)$ be defined as in Theorem \ref{measure}. Then there exists a sequence $\{w_n\}^{\infty}_{n=0}\subset C^{\infty}(\Omega)\cap\BV(\Omega)_p$ such that
\[w_{n}\rightarrow w \mbox{~in~}L_{p}(\Omega)\mbox{~~and~~}\int_{\Omega}(z,Dw_n)\rightarrow \int_{\Omega}(z,Dw).\]
\end{theorem}
\begin{proof}
For any $\epsilon > 0$, take an open set $A \subset \Omega$ such that
\[\int_{\Omega\setminus A}|Dw|<\epsilon,\]
and let $g\in C^{\infty}_{0}(\Omega)$ be such that $0\leq g(x) \leq 1$ in $\Omega$ and $g(x)\equiv 1$ in $A$. We can find a sequence $\{w_{n}\}\subset C^{\infty}(\Omega)\cap \BV(\Omega)_p$(see \cite{Giusti}) such that:
\[w_{n}\rightarrow w~\mbox{in~}L_{p}(\Omega)~\mbox{and}~\limsup_{n\rightarrow \infty}\int_{\Omega\setminus A}|\nabla w_{n}|~dx\leq \int_{ \Omega\setminus A}|Dw|.\]
Then
\begin{eqnarray*}|\int_{\Omega}(z,Dw_n)-\int_{\Omega}(z,Dw)|&\leq& |\langle (z,Dw_n),g\rangle-\langle (z,Dw),g\rangle|\\
&~&+\int_{\Omega}|(z,Dw_n)|(1-g)+\int_{\Omega}|(z,Dw)|(1-g),
\end{eqnarray*}
where
\[\lim_{n\rightarrow \infty}\langle (z,Dw_n),g\rangle =\langle (z,Dw),g\rangle,\]
\[\limsup_{n\rightarrow \infty}\int_{\Omega}|(z,Dw_n)|(1-g)\leq \|z\|_{L_{\infty}}\limsup_{n\rightarrow \infty}\int_{\Omega\setminus A}|Dw_n|<\epsilon \|z\|_{L_{\infty}},\]
\[\int_{\Omega}|(z,Dw)|(1-g)\leq \int_{\Omega\setminus A}|(z,Dw)|<\epsilon \|z\|_{L_{\infty}}.\]
So the theorem is proved, as $\epsilon$ is arbitrary.
\end{proof}
\begin{theorem}
\label{trace}
There exists a linear operator $\gamma:X(\Omega)_{q}\mapsto L_{\infty}(\partial \Omega)$ such that
\begin{enumerate}
\item $\|\gamma(z)\|_{L_{\infty}(\partial \Omega)}\leq \|z\|_{L_{\infty}(\Omega)}$.
\item $\gamma(z)(x)=z(x)\cdot \nu(x)$~$H^{d-1}$-a.e.~on $\partial\Omega$~for~$z\in C^1(\bar{\Omega};\R^d)$.
\item $\langle z, w \rangle_{\partial \Omega}:=\int_{\Omega}w\Div(z)~dx+\int_{\Omega}(z,Dw)=\int_{\partial \Omega}\gamma(z)tr(w)~dH^{d-1}$~~for~any~$w\in\BV(\Omega)_p$, where $tr(w)\in L_1(\partial \Omega)$ is the trace of $w$ on $\partial \Omega$.
\end{enumerate}
\end{theorem}
\begin{proof}
Let $\langle z, w \rangle_{\partial \Omega}:=\int_{\Omega}w\Div(z)~dx+\int_{\Omega}(z,Dw)$, first we want to show 
$\langle z, w_1 \rangle_{\partial \Omega}=\langle z, w_2 \rangle_{\partial \Omega}$ for any \(tr(w_1)=tr(w_2)\) and $w_1,w_2\in \BV(\Omega)_{p}$.
We can find a sequence of functions $\{g_n\}\subset C^{\infty}_{0}(\Omega)$ such that $g_n\rightarrow w_1-w_2$ in $L_p$ and $\int_{\Omega}(z,Dg_n)\rightarrow \int_{\Omega}(z,D(w_1-w_2))$. Then we have:
\begin{eqnarray*}\langle z, w_1-w_2 \rangle_{\partial \Omega}&=&\int_{\Omega}(w_1-w_2)\Div(z)~dx+\int_{\Omega}(z,D(w_1-w_2))\\
&=&\lim_{n\rightarrow \infty}~\{\int_{\Omega}g_n\Div(z)~dx+\int_{\Omega}(z,Dg_n)\}=0.
\end{eqnarray*}
So $\langle z,w_1 \rangle_{\partial \Omega}=\langle z, w_2\rangle_{\partial \Omega}$.

Now we want to show $|\langle z,w \rangle_{\partial \Omega}|\leq \|z\|_{L_{\infty}}\int_{\partial \Omega}|tr(w)|~dH^{d-1}$.
For bounded Lipschitz domain $\Omega$, given any $u\in L_1(\partial \Omega)$ and $\epsilon>0$, we can find a function $w\in W^{1}(L_{1}(\Omega))$ such that
\[tr(w)=u,~~\int_{\Omega}|\nabla w|~dx\leq \int_{\partial \Omega}|u|~dH^{d-1}+\epsilon,~~w(x)=0~\mbox{for}~x\in \Omega_{\e},\]
where $\Omega_{\epsilon}:=\{x\in \Omega:\dist(x,\partial \Omega)>\epsilon\}$. Then for any $tr(w)\in L_{1}(\Omega)$, we can find $\hat{w}\in W^{1}(L_1(\Omega))$ such that $tr(\hat{w})=tr(w)$ with the above properties. So
\begin{eqnarray*}
|\langle z,w\rangle_{\partial \Omega}|&=&|\langle z,\hat{w}\rangle_{\partial \Omega}|\\
&\leq& |\int_{\Omega}\hat{w}\Div(z)~dx|+\|z\|_{L_{\infty}}\int_{\Omega}|D\hat{w}|\\
&\leq& |\int_{\Omega\setminus \Omega_{\epsilon}}\hat{w}\Div(z)~dx|+\|z\|_{L_{\infty}}\{\int_{\partial \Omega}|tr(w)|~dH^{d-1}+\epsilon\}.
\end{eqnarray*}
Since $\lim_{\epsilon\rightarrow 0}\int_{\Omega\setminus \Omega_{\epsilon}}\hat{w}\Div(z)~dx=0$, let $\epsilon$ goes to 0, we get
\begin{equation}
\label{trace-bound}
|\langle z,w\rangle_{\partial \Omega}|\leq \|z\|_{L_{\infty}}\int_{\partial \Omega}|tr(w)|~dH^{d-1}.
\end{equation}

Now given a fixed $z\in X(\Omega)_q$, we can define the linear functional $F_{z}:L_{1}(\partial \Omega)\mapsto \R$ by
\[F_{z}(u):=\langle z, w\rangle_{\partial \Omega},\]
where $tr(w)=u$. From (\ref{trace-bound}), we know $|F_{z}(u)|\leq \|z\|_{L_\infty(\Omega)}\|u\|_{L_1(\partial \Omega)}$. By Rieze Representation theorem, there exists $\gamma(z)\in L_{\infty}(\partial \Omega)$ such that
\[F_{z}(u)=\int_{\partial \Omega}\gamma(z)u~dH^{d-1}.\]
So $\langle z,w\rangle_{\partial \Omega}=\int_{\partial \Omega}\gamma(z)tr(w)~dH^{d-1}$ and $\|\gamma(z)\|_{L_{\infty}(\partial \Omega)}\leq \|z\|_{L_{\infty}(\Omega)}$.

When $z\in C^{1}(\bar{\Omega};\R^{d})$, $\langle z, w\rangle_{\partial \Omega}=\int_{\Omega}\Div(wz)~dx=\int_{\partial \Omega}tr(w)z\cdot \nu~dH^{d-1}$. So $\gamma(z)=z\cdot \nu~H^{d-1}$-a.e. on $\partial \Omega$.
\end{proof}
Thus the function $\gamma(z)$ is a weakly defined trace on $\partial \Omega$ of the normal component of $z$, we shall denote $\gamma(z)$ by $[z,\nu]$. In this way, the Neumann boundary condition can be expressed as $[z,\nu]=0~H^{d-1}$-a.e. on $\partial \Omega$.
\subsection{Definition of the Operator $-\Div(\frac{D u}{|D u|})$}
Let $A: X\mapsto 2^{X^*}$ be a \emph{multivalued mapping} defined on a Banach space $X$, i.e., $A$ assigns to each point $u\in X$ a subset $Au$ of $X^*$, where $X^*$ is the dual space of $X$. In this paper we will simply call such mapping an \emph{operator}.
\begin{enumerate}
\item The set $D(A):=\{u\in X: Au\neq \varnothing\}$ is called the \emph{effective domain} of $A$. When $D(A)\neq \emptyset$, we say $A$ is \emph{proper}.
\item The set $R(A):=\bigcup_{u\in X}Au$ is called the \emph{range} of $A$.
\item The set $G(A):=\{(u,v)\in X\times Y:u\in D(A),v\in Au\}$ is called the \emph{graph} of $A$. In this paper, we briefly write $(u,v)\in A$ instead of $(u,v)\in G(A)$ and we will identify an operator $A$ with its graph $G(A)$.
\item An operator $A$ is called a \emph{monotone operator}, if $\langle v_1-v_2,u_1-u_2 \rangle \geq 0$ for $(u_1,v_1),(u_2,v_2)\in A$.
\item A monotone operator $A$ is called a \emph{maximal monotone operator}, if for any monotone operator $B$ that $A\subset B$, we have $A=B$.
\end{enumerate}
\begin{lemma}
\label{nonexpansive}
Let $X$ be a Hilbert space and $A:D(A)\subset X \mapsto 2^X$. If $A$ is a monotone operator, then for any $\lambda>0$, $(I+\lambda A)^{-1}:X \mapsto D(A)$ is nonexpansive, i.e. for $v_1,v_2\in D((I+\lambda A)^{-1})$, we have $\|(I+\lambda A)^{-1}(v_1)-(I+\lambda A)^{-1}(v_2)\|_{X}\leq \|v_1-v_2\|_{X}$.
\end{lemma}
\begin{proof}
Let $v_1\in(I+\lambda A)(u_1)$ and $v_2\in(I+\lambda A)(u_2)$. Then $w_1=\frac{1}{\lambda}(v_1-u_1)\in A(u_1)$ and $w_2=\frac{1}{\lambda}(v_2-u_2)\in A(u_2)$. By the fact $\langle w_1-w_2,u_1-u_2 \rangle \geq 0$,  we have:
\begin{eqnarray*}\|v_1-v_2\|^2_{X}&=&\langle (I+\lambda A)(u_1)-(I+\lambda A)(u_2),(I+\lambda A)(u_1)-(I+\lambda A)(u_2) \rangle\\
&=&\|u_1-u_2\|^2_{X}+\lambda^2\|w_1-w_2\|^2_{X}+2\lambda \langle w_1-w_2,u_1-u_2 \rangle \\
&\geq&\|u_1-u_2\|^2_{X}.
\end{eqnarray*}
\end{proof}
Now we introduce the following operator $\mathcal{A}$ on $L_2(\Omega)$:
\begin{definition}
$v\in \mathcal{A}(u)$ means:
\[\begin{array}{l}
\exists~z\in X(\Omega)_2 \mbox{~with~} \|z\|_{L_\infty} \leq 1, v=-\Div(z) \mbox{~in~} \mathcal{D'}(\Omega) \mbox{~such that}\\
\int_\Omega (\phi-u)v~dx=\int_\Omega (z,D\phi)-|u|_{\BV} \mbox{~for any~} \phi\in \BV(\Omega)\cap L_2(\Omega).
\end{array}\]
\end{definition}
\noindent Let's recall the set $S$ that we used for deriving Euler-Lagrange Equation:
\[S:=\{u\in C^1(\bar{\Omega})~:~\frac{\partial u}{\partial \nu}=0 \mbox{~for~} x \in \partial \Omega \mbox{~and~} \nabla u \neq 0 \mbox{~for all~} x \in \Omega\}.\]
For $u\in S$, $\mathcal{A}(u)=-\Div(\frac{\nabla u}{|\nabla u|})$. Hence the operator $\mathcal{A}$ can be viewed as generalization of $-\Div(\frac{\nabla u}{|\nabla u|})$. \emph{Formally}, we can write $\mathcal{A}(u)=-\Div(\frac{D u}{|D u|})$ for $u\in \BV(\Omega)\cap L_2(\Omega)$.

To associate the operator $\mathcal{A}$ with our minimization problem (\ref{equ:L2-BV}), we need to introduce the concept about \emph{subdifferential}. Let $L:X\mapsto [-\infty,+\infty]$ be a functional on a real Banach space $X$. The functional $u^*$ in $X^*$ is called a \emph{subgradient} of $L$ at the point $u$ if and only if $L(u)\neq \pm \infty$ and
\[L(w)\geq L(u)+\langle u^*, w-u \rangle \mbox{~~for any } w\in X.\]
For each $u\in X$, the set:
\[\partial L(u):=\{u^* \in X^*:u^* \mbox{~is a subgradient of~} L \mbox{~at~}u\}\]
is called the \emph{subdifferential} of $L$ at $u$. Thus $\partial L$ is a multivalued mapping defined on $X$ and $\partial L(u)=\emptyset$ if $L(u)=\pm \infty$.
\begin{theorem}
\label{semi}
Let $L:X\mapsto (-\infty,+\infty]$ be a proper convex and lower semi-continuous functional on the real Banach space $X$, then the subdifferential $\partial L:X\mapsto 2^{X^*}$ is maximal monotone.
\end{theorem}
\begin{proof}
See the fundamental paper by R. T. Rockafellar(\cite{Rockafellar}).
\end{proof}
\begin{theorem}
Let $X$ be a Hilbert space and $A:D(A)\subset X\mapsto 2^X$. Then the following two statements are equivalent:
\begin{enumerate}
\item $A$ is a monotone operator and $R(I+A)=X$
\item $A$ is a maximal monotone operator.
\end{enumerate}
\end{theorem}
\begin{proof}~\\
$1. \Rightarrow 2.$\\
We only need to show that: if for any $v\in A(u)$, we have $\langle v-v_0, u-u_0\rangle \geq 0$, then $v_0\in A(u_0)$.
Since $R(I+A)=X$, we can find $u_1\in X$ such that $u_0+v_0= u_1+v_1$ for some $v_1\in A(u_1)$. So $\langle v_1-v_0,u_1-u_0 \rangle = -\|u_1-u_0\|_{X}^{2}\geq 0$. Consequently, we have $u_0=u_1$, so $v_0\in A(u_0)$.\\
$2. \Rightarrow 1.$ See the fundamental paper by G. Minty(\cite{Minty}).
\end{proof}

\begin{theorem}
The operator $\mathcal{A}$ defined above is a maximal monotone operator on $L_2(\Omega)$.
\end{theorem}
\noindent To prove this theorem, we need to introduce a p-Laplace type operator $A_p$ defined on $L_2(\Omega)$:
\[\begin{array}{c}(u,v)\in A_p \mbox{~~if and only if~}u\in W^{1}(L_p(\Omega))\cap L_2(\Omega), v\in L_2(\Omega)\mbox{~and}\\
\int_{\Omega}v\phi~dx=\int_{\Omega}|\nabla u|^{p-2}\nabla u\cdot \nabla \phi~dx~\mbox{for any~}\phi\in W^{1}(L_p(\Omega))\cap L_2(\Omega).
\end{array}\]
From the definition of $A_p$, we can see that $v=A_p(u)=-\Div(|\nabla u|^{p-2}\nabla u)$ in $\mathcal{D'}(\Omega)$ for $u\in D(A_p)$.
\begin{lemma}
$A_p$ is a monotone operator and $R(I+A_p)=L_2(\Omega)$, i.e. $A_p$ is maximal monotone on $L_2(\Omega)$.
\end{lemma}
\label{lemma:A_p}
\begin{proof}~\\
(i) $A_p$ is monotone: let $(u_1,v_1),(u_2,v_2)\in A_p$, then
\begin{eqnarray*}\langle v_1-v_2, u_1-u_2 \rangle&=&\int_\Omega v_1 u_1 + v_2 u_2 - v_1 u_2 - v_2 u_1~dx\\
&=&\int_\Omega |\nabla u_1|^{p-2}\nabla u_1\cdot \nabla u_1 + |\nabla u_2|^{p-2}\nabla u_2 \cdot \nabla u_2\\
& & - |\nabla u_1|^{p-2}\nabla u_1 \cdot \nabla u_2 - |\nabla u_2|^{p-2}\nabla u_2 \cdot \nabla u_1~dx\\
&\geq&\int_\Omega |\nabla u_1|^p+|\nabla u_2|^p-|\nabla u_1|^{p-1}|\nabla u_2|-|\nabla u_2|^{p-1}|\nabla u_1|~dx\\
&=&\int_\Omega (|\nabla u_1|^{p-1}-|\nabla u_2|^{p-1})(|\nabla u_1|-|\nabla u_2|)~dx\\
&\geq&0.
\end{eqnarray*}
(ii) $R(I+A_p)=L_2(\Omega)$: for any $v\in L_2(\Omega)$, we need to find $u\in D(A_p)$ such that $u+A_p(u)\ni v$. That is to say:
 given $v\in L_2(\Omega)$, we need to find $u\in  W^{1}(L_p(\Omega))\cap L_2(\Omega)$ such that
\begin{equation}
\int_\Omega (v-u)\phi~dx=\int_\Omega |\nabla u|^{p-2}\nabla u\cdot \nabla \phi~dx\label{A_p:weak}
\end{equation}
for any $\phi\in W^{1}(L_p(\Omega))\cap L_2(\Omega)$.
It is actually a weak solution of the Euler-Lagrange equation for the following minimization problem:
\begin{equation}\min_{u\in  W^{1}(L_p(\Omega))\cap L_2(\Omega)} T(u)\label{A_p:min}\end{equation}
where $T(u):=\frac{1}{2}\|v-u\|^2_{L_2}+\frac{1}{p}\|\nabla u\|^p_{L_p}$.
We will first prove there exists a minimizer for (\ref{A_p:min}). Then we will show that such a minimizer is a solution for (\ref{A_p:weak}). Select a minimizing sequence $\{u_k\}_{k=1}^{\infty}$ for (\ref{A_p:min}). Without loss of generality, we can assume $\int_\Omega (v-u_k)~dx=0$ for every $u_k$, otherwise if $\frac{1}{|\Omega|}\int_\Omega (v-u_k)~dx=c\neq 0$, then $\|v-(u_k+c)\|^2_{L_2}<\|v-u_k\|^2_{L_2}$. Let $m=\inf_{u\in  W^{1}(L_p(\Omega))\cap L_2(\Omega)}T(u)<\infty$. Since $T(u_k)\rightarrow m$ and $\|\nabla u_k\|_{L_p}\leq pT(u_k)$, we have $\sup_{k} \|\nabla u_k\|_{L_p}<\infty$. Since $\int_{\Omega}u_k-u_1~dx=0$, applying Poincare inequality, we have:

\begin{eqnarray*}\|u_k\|_{L_p}&\leq&\|u_k-u_1\|_{L_p}+\|u_1\|_{L_p}\leq C\|\nabla u_k-\nabla u_1\|_{L_p}+\|u_1\|_{L_p}\\
&~&\leq C\|\nabla u_k\|_{L_p}+C\|\nabla u_1\|_{L_p}+\|u_1\|_{L_p}.
\end{eqnarray*}
So $\{u_k\}_{k=1}^{\infty}$ is bounded in $W^{1}(L_p(\Omega))$. In addition, since $\|u_k\|_{L_2}\leq \sqrt{2T(u_k)}+\|v\|_{L_2}$, we also have $\{u_k\}_{k=1}^{\infty}$ bounded in $L_2(\Omega)$. Consequently there exists a subsequence $\{u_{k_j}\}_{j=1}^{\infty}\subset \{u_k\}_{k=1}^{\infty}$ and a function $u\in W^{1}(L_p(\Omega))$ such that $u_{k_j}\rightharpoonup u$ weakly in $W^{1}(L_p(\Omega))$ and in $L_2(\Omega)$. So $T(u)\leq \liminf_{j\rightarrow \infty}T(u_{k_j})$. By the fact that $\{u_{k_j}\}^\infty_{j=1}$ is a minimizing sequence, we have $T(u)\leq m$. But from the definition of $m$, $m\leq T(u)$. Consequently $u$ is indeed a minimizer.
Now we will show that $u$ is a solution for (\ref{A_p:weak}).
\begin{eqnarray*}\delta T(u;\phi)&=&\lim_{\epsilon \rightarrow 0}\frac{1}{\epsilon}\{T(u+\epsilon \phi)-T(u)\}\\
&=&\lim_{\epsilon \rightarrow 0}\frac{1}{\epsilon}\{\frac{1}{2}(\|v-(u+\epsilon \phi)\|^2_{L_2}-\|v-u\|^2_{L_2})+\frac{1}{p}(\|\nabla (u+\epsilon \phi)\|^p_{L_p}-\|\nabla u\|^p_{L_p})\}\\
&=&\lim_{\epsilon \rightarrow 0}\frac{1}{\epsilon}\int_{\Omega}(u-v)\epsilon \phi+|\nabla u|^{p-2}\nabla u\cdot (\epsilon \nabla \phi)+O(\epsilon^2)~dx\\
&=&\int_\Omega (u-v)\phi+|\nabla u|^{p-2}\nabla u\cdot \nabla \phi~dx.
\end{eqnarray*}
The necessary condition for $u$ to be a minimizer of $T(u)$ is $\delta T(u;\phi)=0$, which is (\ref{A_p:weak}).
\end{proof}
\noindent Roughly speaking, we want to see \lq\lq{$\mathcal{A}(u)=\lim_{p\rightarrow 1}A_p(u)$}\rq\rq. It is easy to show that $\mathcal{A}$ is monotone, to prove it is a maximal monotone operator on $L_2(\Omega)$, we also need the range condition: $R(I+\mathcal{A})=L_2(\Omega)$. We need the following two lemmas:
\begin{lemma}
$\mathcal{A}$ is a monotone operator on $L_2(\Omega)$.
\end{lemma}
\begin{proof}
Let $(u_1,v_1),(u_2,v_2)\in \mathcal{A}$. We have:
\begin{eqnarray*}\langle v_1-v_2,u_1-u_2 \rangle &=&\langle v_1,u_1 \rangle+ \langle v_2,u_2 \rangle- \langle v_1,u_2 \rangle - \langle v_2,u_1 \rangle\\
&=&\int_{\Omega} v_1 u_1+v_2 u_2-v_1 u_2 - v_2 u_1~dx\\
&=&\int_{\Omega} (z_1, Du_1)+(z_2,Du_2)-(z_1,Du_2)-(z_2,Du_1)\\
&=&|u_1|_{\BV}+|u_2|_{\BV}-\int_{\Omega}(z_1,Du_2)-\int_{\Omega}(z_2,Du_1)
\end{eqnarray*}
Since $\int_{\Omega}(z_1,Du_2)+(z_2,Du_1)\leq |u_2|_{\BV}+|u_1|_{\BV}$, we get $\langle v_1-v_2,u_1-u_2 \rangle \geq 0$, which means $\mathcal{A}$  is monotone.
\end{proof}
\begin{lemma}
$R(I+\mathcal{A})=L_2(\Omega)$
\end{lemma}
\begin{proof}
We only need to show that: for any $v\in L_{2}(\Omega)$, there exists $u\in \BV(\Omega)\cap L_2(\Omega)$ such that $(u,v-u)\in \mathcal{A}$. 

By lemma \ref{lemma:A_p}, we know that: given $v\in L_{2}(\Omega)$, for any $p>1$, there is $u_p\in W^{1}(L_p(\Omega))\cap L_2(\Omega)$ such that $(u_p,v-u_p)\in A_p$. Hence we have
\begin{equation}
\label{Range:1}
\int_{\Omega}(v-u_p)\phi~dx=\int_{\Omega}|\nabla u_p|^{p-2}\nabla u_p\cdot \nabla \phi~dx\end{equation}
for every $\phi \in W^{1}(L_p(\Omega))\cap L_{2}(\Omega)$. 

Since $A_p$ is maximal monotone, by lemma \ref{nonexpansive}, and noticing $0\in A_p(0)$, we have
\begin{equation}
\label{Range:2}
\|u_p\|_{L_2}=\|(I+A_p)^{-1}v\|_{L_2}\leq \|v\|_{L_2}.\end{equation}
Take $\phi=u_p$ and combine (\ref{Range:2}), we get the estimate:
\begin{equation}
\label{Range:3}
\int_{\Omega}|\nabla u_p|^p~dx=\int_{\Omega}(v-u_p)u_p~dx\leq \|v-u_p\|_{L_2}\|u_p\|_{L_2}\leq 2\|v\|^2_{L_2}=M,
\end{equation}
for any $p>1$. 

By using Holder inequality we also have:
\[\int_{\Omega}|\nabla u_p|~dx\leq (\int_{\Omega}1~dx)^{\frac{1}{q}}(\int_{\Omega}|\nabla u_p|^{p}~dx)^{\frac{1}{p}}=|\Omega|^{\frac{1}{q}}M^{\frac{1}{p}}\leq \max\{|\Omega|M,|\Omega|,M,1\}=M_0.\]
Hence $\{u_p\}_{p>1}$ is bounded in $W^{1}(L_1(\Omega))$ and we may extract a subsequence such that $u_p$ converges in $L_1(\Omega)$ and almost everywhere to some $u\in L_1(\Omega)$ as $p\rightarrow 1+$. 

By Fatou's lemma, we have
\[\int_{\Omega} u^2~dx\leq \liminf_{p\rightarrow 1+} \int_{\Omega} u_p^2~dx\leq \|v\|^2_{L_2}.\]
By lemma \ref{semi-contin}, we get
\[\int_{\Omega} |D u| \leq \liminf_{p\rightarrow 1+} \int_{\Omega}|\nabla u_p|~dx\leq M_0,\]
so we have that $u\in\BV(\Omega)\cap L_2(\Omega)$. Since $\{u_p\}_{p>1}$ is bounded in $L_2(\Omega)$, without loss of generality, we can assume $u_p \rightharpoonup u$ weakly in $L_2(\Omega)$ as $p\rightarrow 1+$.\\
Now we want to show that $\{|\nabla u_p|^{p-2}\nabla u_p\}_{p>1}$ is weakly relatively compact in $L_{2}(\Omega;\R^d)$. First we need to obtain the following two estimates:
\[\int_{\Omega}|\nabla u_p|^{p-1}~dx\leq (\int_{\Omega}|\nabla u_p|^p~dx)^{\frac{p-1}{p}}|\Omega|^{\frac{1}{p}}\leq M^{\frac{p-1}{p}}|\Omega|^{\frac{1}{p}}\leq M_0\]
and for any measurable subset $E\subset \Omega$,
\[|\int_E |\nabla u_p|^{p-2}\nabla u_p~dx|\leq \int_E |\nabla u_p|^{p-1}~dx\leq M^{\frac{p-1}{p}}|E|^{\frac{1}{p}}\leq \max \{M,1\}|E|^{\frac{1}{2}},~\mbox{for~}|E|<1\mbox{~and~}1<p<2.\]
Hence $\{|\nabla u_p|^{p-2}\nabla u_p\}_{p>1}$ is bounded and equiintegrable in $L_{1}(\Omega;\R^d)$, and consequently weakly relatively compact in $L_1(\Omega;\R^d)$. Thus without loss of generality, we can assume: there exists $z\in L_1(\Omega;\R^d)$ such that,
\[|\nabla u_p|^{p-2}\nabla u_p \rightharpoonup z \mbox{~~as~~} p\rightarrow 1+, \mbox{~weakly in~} L_1(\Omega;\R^d).\]
Take $\phi\in C^{\infty}_{0}(\Omega)$ in (\ref{Range:1}) and let $p\rightarrow 1+$, we obtain
\[\int_{\Omega}(v-u)\phi=\int_{\Omega}z\cdot \nabla \phi,\]
which means $v-u=-\Div(z)$ in $\mathcal{D}'(\Omega)$.\\
Now we need to prove $\|z\|_{L_\infty}\leq 1$. For any $k>0$, let $B_{p,k}:=\{x\in \Omega:|\nabla u_p(x)|>k\}$. By (\ref{Range:3}), we have:
\[k^p|B_{p,k}|\leq \int_{B_{p,k}}|\nabla u_p|^p~dx\leq \int_{\Omega}|\nabla u_p|^p~dx\leq M.\]
Hence $|B_{p,k}|\leq \frac{M}{k^p}\leq \max\{\frac{M}{k},\frac{M}{k^2}\}$ for every $1<p<2$ and $k>0$.

As above, there is some $g_k\in L_1(\Omega;\R^d)$ such that $|\nabla u_p|^{p-2}\nabla u_p \Chi_{B_{p,k}} \rightharpoonup g_k$ weakly in $L_1(\Omega;\R^d)$ as $p\rightarrow 1+$.

Now for any $\phi\in L_{\infty}(\Omega;\R^d)$ with $\|\phi\|_{L_{\infty}}\leq 1$, we can prove that
\[|\int_{\Omega}|\nabla u_p|^{p-2}\nabla u_p\Chi_{B_{p,k}}\cdot \phi~dx| \leq \int_{B_{p,k}} |\nabla u_p|^{p-1}~dx\leq \frac{1}{k}\int_{B_{p,k}}|\nabla u_p|^p~dx\leq \frac{M}{k}.\]
Since
\[|\int_{\Omega}g_k\cdot \phi~dx|\leq |\int_{\Omega}|\nabla u_p|^{p-2}\nabla u_p\Chi_{B_{p,k}}\cdot \phi~dx|+|\int_{\Omega}(|\nabla u_p|^{p-2}\nabla u_p\Chi_{B_{p,k}}-g_k)\cdot \phi~dx|,\]
let $p\rightarrow 1+$, we get:
$|\int_{\Omega}g_k\cdot \phi~dx|\leq \frac{M}{k}$. So $\int_{\Omega}|g_k|~dx\leq \frac{M}{k}$ for every $k>0$. Hence $g_k\rightarrow 0$ in $L_1(\Omega;\R^d)$ and a.e..

Since we have:
\[|~|\nabla u_p|^{p-2}\nabla u_p \Chi_{\Omega \setminus B_{p,k}}|\leq k^{p-1}\leq \max \{k,k^2\}~\mbox{~for~~}1<p<2,\]
there exists $f_k \in L_{\infty}(\Omega;\R^d)$ and we can assume $|\nabla u_p|^{p-2}\nabla u_p \Chi_{\Omega \setminus B_{p,k}}\rightharpoonup f_k$ weakly in $L_{\infty}(\Omega;\R^d)$ when $p\rightarrow 1+$. Thus $\|f_k\|_{L_\infty}\leq \liminf_{p\rightarrow 1+}\|~|\nabla u_p|^{p-2}\nabla u_p \Chi_{\Omega \setminus B_{p,k}}\|_{L_\infty}\leq 1$. 

Since for any $k>0$, we can write $z=f_k+g_k$, then we have $z-f_k=g_k\rightarrow 0$ a.e.. By the fact $\|f_k\|_{L_\infty}\leq 1$, we get $\|z\|_{L_\infty}\leq 1$.
\\
To prove $(u,v-u)\in \mathcal{A}$, we also need to show
\[\int_\Omega (\phi-u)(v-u)~dx=\int_\Omega (z,D\phi)-|u|_{\BV} \mbox{~~for any~}\phi\in W^{1}(L_1(\Omega))\cap L_2(\Omega).\]
For any $\phi\in W^{1}(L_1(\Omega))\cap L_2(\Omega)$, let $\phi_n\in C^{\infty}(\bar{\Omega})$ be such that $\phi_n\rightarrow \phi$ in $W^{1}(L_1(\Omega))\cap L_2(\Omega)$ as $n \rightarrow \infty$. Using $\phi_n-u_p$ as a test function in (\ref{Range:1}), we get
\[\int_{\Omega}(v-u_p)(\phi_n-u_p)~dx=\int_{\Omega}|\nabla u_p|^{p-2}\nabla u_p\cdot \nabla(\phi_n-u_p)~dx.\]
Hence,
\begin{equation}
\label{Range:4}
\int_{\Omega}(v-u_p)(\phi_n-u_p)~dx+\int_{\Omega}|\nabla u_p|^p~dx=\int_{\Omega}|\nabla u_p|^{p-2}\nabla u_p\cdot \nabla \phi_n~dx.
\end{equation}
Since $u_p \rightharpoonup u$ weakly in $L_2(\Omega)$ as $p \rightarrow 1+$, we have $\|u\|_{L_2}\leq \liminf_{p \rightarrow 1+}\|u_p\|_{L_2}$, $\int_{\Omega}(v-u_p)\phi_n~dx\rightarrow \int_{\Omega}(v-u)\phi_n~dx$ and $\int_{\Omega} vu_p~dx\rightarrow \int_{\Omega}vu~dx$. 

And since we have $\|\nabla u_p\|_{L_1}\leq \|\nabla u_p\|_{L_p}|\Omega|^{1-\frac{1}{p}}$ and $|u|_{\BV} \leq \liminf_{p \rightarrow 1+} \|\nabla u_p\|_{L_1}$, then
\[|u|_{\BV}\leq \liminf_{p \rightarrow 1+} \|\nabla u_p\|_{L_p}|\Omega|^{1-\frac{1}{p}}\leq (\liminf_{p\rightarrow 1+} \|\nabla u_p\|_{L_p})(\lim_{p \rightarrow 1+} |\Omega|^{1-\frac{1}{p}})=\liminf_{p\rightarrow 1+} \|\nabla u_p\|_{L_p}.\]
So
\begin{eqnarray*}|u|_{\BV}&\leq&(\liminf_{p \rightarrow 1+}\|\nabla u_p\|^p_{L_p})(\limsup_{p \rightarrow 1+}\|\nabla u_p\|^{1-p}_{L_p})\\
&\leq&(\liminf_{p \rightarrow 1+}\|\nabla u_p\|^p_{L_p})(\lim_{p \rightarrow 1+}M^{\frac{1-p}{p}})=\liminf_{p \rightarrow 1+}\|\nabla u_p\|^p_{L_p}.
\end{eqnarray*}
Combining (\ref{Range:4}), we get:
\[\int_{\Omega}(v-u)(\phi_n-u)~dx+|u|_{\BV}\leq \int_{\Omega}z\cdot \nabla \phi_n~dx.\]
Letting $n \rightarrow \infty$, we get:
\[\int_{\Omega}(v-u)(\phi-u)~dx+|u|_{\BV}\leq \int_{\Omega}z\cdot \nabla \phi~dx.\]
By Theorem \ref{dense}, for any $\phi\in \BV(\Omega)\cap L_2(\Omega)$, we can find a sequence $\{\phi_n\}\subset W^{1}(L_1(\Omega))\cap L_2(\Omega)$ such that:
\[\phi_{n}\rightarrow \phi \mbox{~in~}L_{2}(\Omega)\mbox{~~and~~}\int_{\Omega}(z,D\phi_n)\rightarrow \int_{\Omega}(z,D\phi).\]
So we can claim that
\[\int_{\Omega}(v-u)(\phi-u)~dx+|u|_{\BV}\leq \int_{\Omega}(z,D\phi)\]
holds for any $\phi\in \BV(\Omega)\cap L_2(\Omega)$. Consequently, we have:
\[\int_{\Omega}(v-u)(\phi-u)~dx+|u|_{\BV}= \int_{\Omega}(z,D\phi).\]
So the theorem is proved.
\end{proof}
\noindent Recall:
\[J(u):=\left \{ \begin{array}{ll}|u|_{\BV}&u\in \BV(\Omega)\cap L_2(\Omega)\\+\infty&u\in L_2(\Omega) \backslash  (\BV(\Omega)\cap L_2(\Omega))\end{array} \right.\]
$J$ is a proper convex functional defined on $L_2(\Omega)$. Since $(L_2(\Omega))^*=L_2(\Omega)$, we have $\partial J\subset L_2(\Omega)\times L_2(\Omega)$.
Now we can present the fundamental theorem in this subsection:
\begin{theorem}
\label{th:equiv}
$\partial J =\mathcal{A}$.
\end{theorem}
\begin{proof}
Since the functional $J(u)$ is proper convex and lower semi-continuous,  by Theorem \ref{semi}, we know $\partial J$ is maximal monotone. It is easy to see $\mathcal{A}\subset \partial J$. As we proved above, $\mathcal{A}$ is maximal monotone. So $\partial J=\mathcal{A}$.

\end{proof}
\begin{remark}
The results in this section also holds for $\Omega=\mathbb{R}^d$.
\end{remark}

\subsection{Characterization of the Pair $(L_2(\Omega),\BV(\Omega))$}
From the definition of subdifferential, we get the following \emph{Minimum Principle}:
\begin{theorem}
Let $L:X\mapsto (-\infty,+\infty]$ be a proper functional on the real Banach space $X$. Then $u^*$ is a minimizer for $L(u)$ if and only if $0\in \partial L(u^*)$.
\end{theorem}
Since $\partial T(u)=t \partial J(u)-(f-u)$, the necessary and sufficient condition for $u_t$ to be a minimizer of problem (\ref{equ:L2-BV}) is:
\begin{equation}
\label{equ:iff}
t\partial J(u_t)-(f-u_t)\ni 0.\end{equation}
Since we have shown in Theorem \ref{th:equiv} that $\mathcal{A}=\partial J$, (\ref{equ:iff}) can be rewritten as:
\[t\mathcal{A}(u_t)-(f-u_t)\ni 0.\]
\begin{theorem}
\label{th:char}
The following assertions are equivalent:
\begin{enumerate}
\item $u_t$ is a minimizer for problem (\ref{equ:L2-BV}).
\item $u_t\in \BV(\Omega)_2$ and there exists $z \in X(\Omega)_2$ with $\|z\|_{L_\infty}\leq 1$, $f-u_t=-t\Div(z)$ such that
\[-\int_\Omega\Div(z)(\phi-u_t)~dx=\int_\Omega(z,D\phi)-|u_t|_{\BV}\]
for any $\phi\in \BV(\Omega)_2$.
\item $u_t\in \BV(\Omega)_2$ and there exists $z \in X(\Omega)_2$ with $\|z\|_{L_\infty}\leq 1$, $f-u_t=-t\Div(z)$ such that
\[\left \{\begin{array}{ll}&\int_\Omega (z,Du_t)=|u_t|_{\BV} \\&[z,\nu]=0\end{array}\right.
\]
\end{enumerate}
\end{theorem}
\begin{proof}~\\
$2.\Rightarrow 1.$
Since $\|z\|_{L_\infty}\leq 1$, we have $\int_\Omega(z,D\phi)\leq |\phi|_{\BV}$. Then:
\[\int_{\Omega}\frac{1}{t}(f-u_t)(\phi-u_t)~dx \leq |\phi|_{\BV}-|u_t|_{\BV}.\]
Since
\begin{eqnarray*}\frac{1}{2t}((f-u_t)^2-(f-\phi)^2)&=&\frac{1}{t}(f-\frac{u_t+\phi}{2})(\phi-u_t)\\
&=&\frac{1}{t}(f-u_t)(\phi-u_t)-\frac{1}{2t}(\phi-u_t)^2\\
&\leq& \frac{1}{t} (f-u_t)(\phi-u_t),
\end{eqnarray*}
we have $\frac{1}{2t}\|f-u_t\|^2_{L_2}-\frac{1}{2t}\|f-\phi\|^2_{L_2}\leq |\phi|_{\BV}-|u_t|_{\BV}$ for any $\phi\in \BV(\Omega)_2$.
This tells us that $u_t$ is a minimizer for problem (\ref{equ:L2-BV}).\\
$3.\Rightarrow 2.$
By Green's formula and the boundary condition $[z,\nu]=0$, we have
\[-\int_\Omega\mbox{div}(z)(\phi-u_t)~dx=\int_\Omega(z,D\phi)-\int_\Omega(z,Du_t).\]
Plug in $\int_\Omega (z,Du_t)=|u_t|_{\BV}$. We get $-\int_\Omega\Div(z)(\phi-u_t)~dx=\int_\Omega(z,D\phi)-|u_t|_{\BV}$.\\
$1.\Rightarrow 3.$ Since $u_t$ is a minimizer for problem (\ref{equ:L2-BV}), we have $t \mathcal{A}(u_t)\ni (f-u_t)$. So there exists $z \in X(\Omega)_2$ with $\|z\|_{L_\infty}\leq 1$, $v_t=\frac{1}{t}(f-u_t)=-\Div(z)$ such that:
\[\left \{ \begin{array}{l}\int_\Omega v_t\phi~dx=\int_\Omega (z, D\phi) \mbox{~for any~}\phi\in \BV(\Omega)_2
\\ \int_\Omega (z,Du_t)=|u_t|_{\BV}\end{array}\right.\]
So $\langle z, \phi \rangle_{\partial \Omega}=\int_\Omega \phi\Div(z)~dx+\int_\Omega (z, D\phi)=0$ for any $\phi\in \BV(\Omega)_2$. Hence $[z,\nu]=0$.
\end{proof}
\begin{remark}
From Theorem \ref{th:char}, we can see that the Neumann boundary condition is a natural assumption for problem (\ref{equ:L2-BV}) and $\int_\Omega f-u_t~dx=\int_\Omega -t\Div(z)~dx=-\int_{\partial \Omega} t[z,\nu]~dH^{d-1}=0$ will automatically hold.
\end{remark}

In order to go further, we denote  the homogeneous part of $L_p(\Omega)$ by  $L^{p}_{\Box}(\Omega):=\{v\in L_{p}(\Omega):\int_{\Omega}v~dx=0\}$ and define $\dot X(\Omega)_{p}:=\{z\in X(\Omega)_p:[z,\nu]=0\}$.
As a rather deep result of  \cite{Bourgain}, Bourgain and Brezis prove that for every $v\in L^{p}_{\Box}(\Omega)$ with $p\geq d$, there exists  $z\in \dot X(\Omega)_{p}$ such that $v=\mbox{div}(z)$. This result is obviously not true for $p<d$ and so we introduce the
following norm for $v\in L^{p}_{\Box}(\Omega)$ with $1\leq p \leq \infty$:
\[\|v\|_{Y_p}:=\inf~{\{\liminf_{k\rightarrow \infty}{\|z_k\|_{L_{\infty}}}~:~\lim_{k\rightarrow \infty}\|\Div(z_k)-v\|_{L_p}=0,~z_{k}\in \dot{X}(\Omega)_{p}\}}\]
and  the corresponding normed vector space:
\[Y(\Omega)_{p}:=\{v \in L^{p}_{\Box}(\Omega)~:~\|v\|_{Y_p}<\infty\}.\]
Then we have the following characterization for $Y(\Omega)_{p}$:
\begin{theorem}
\label{th:Y_p}
For every $v\in Y(\Omega)_p$, there exists $z\in \dot{X}(\Omega)_p$ such that
\[v=\Div(z), ~~~\|z\|_{L_{\infty}}=\|v\|_{Y_p}.\]
In addition, the unit ball $U_{p}:=\{v\in L^{p}_{\Box}(\Omega):\|v\|_{Y_p}\leq 1\}$ is closed in $L_p$ norm topology.
\end{theorem}
\begin{proof}
From the definition of the $\|\cdot\|_{Y_p}$ norm, we can find a sequence $\{z_k\}\subset \dot{X}(\Omega)_p$ such that:
\[\lim_{k\rightarrow \infty}\|\Div(z_k)-v\|_{L_p}=0,~~~\lim_{k\rightarrow \infty}\|z_k\|_{L_{\infty}}=\|v\|_{Y_p}.\]
Hence $\{\|z_k\|_{L_{\infty}}\}$ is bounded, so up to an extraction, we can find $z\in L_{\infty}(\Omega;\mathbb{R}^d)$ such that $z_k$ converges to $z$ weak-* in $L_{\infty}(\Omega;\mathbb{R}^d)$. Then for every $\phi \in C^{\infty}(\overline{\Omega})$, 
\begin{eqnarray*}\int_{\Omega}v \phi~dx&=&\lim_{k \rightarrow \infty}\int_{\Omega}\Div(z_k)\phi~dx=\lim_{k\rightarrow \infty}-\int_{\Omega}z_{k}\cdot \nabla \phi~dx\\
&=&-\int_{\Omega}z\cdot \nabla \phi~dx=\int_{\Omega}\Div(z)\phi~dx-\int_{\partial \Omega}[z,\nu]\phi~dH^{d-1}.
\end{eqnarray*}
Choose $\phi \in C^{\infty}_{0}(\Omega)$, we get $\mbox{div}(z)=v\in L_p(\Omega)$ in the sense of distribution. So for any $\phi \in C^{\infty}(\overline{\Omega})$, we have:
\[\int_{\partial \Omega}[z,\nu]\phi~dH^{d-1}=0.\]
Hence $[z,\nu]=0$~$H^{d-1}$ a.e. on $\partial \Omega$.

By weak-* lower semi-continuity of $\|\cdot\|_{L_\infty}$, we get:
\[\|z\|_{L_{\infty}}\leq \lim_{k \rightarrow \infty}\|z_k\|_{L_{\infty}}=\|v\|_{Y_p}.\]
By the definition of the $\|\cdot\|_{Y_p}$ norm, we have $\|z\|_{L_{\infty}}\geq \|v\|_{Y_p}$.
So $\|z\|_{L_{\infty}}=\|v\|_{Y_p}$.

Now let $\{v_n\}$ be a sequence in $U_p$ such that $v_n\rightarrow v$ for some $v\in L^{p}_{\Box}(\Omega)$, we want to show $v\in U_{p}$. Since $v_{n}=\Div(z_n)$ with $z_n\in \dot{X}(\Omega)_p$ and $\|v_{n}\|_{Y_p}=\|z_n\|_{L_\infty}$, we have $\|z_n\|_{L_\infty}\leq 1$. Thus we can find $z\in L_{\infty}(\Omega;\R^{d})$ such that, up to an extraction, $z_{n}\rightharpoonup z$ weak-* in $L_{\infty}(\Omega;\R^{d})$. So $\|z\|_{L_\infty}\leq \liminf_{n\rightarrow \infty}\|z_n\|_{L_\infty}\leq 1$. For any $\phi \in C^{\infty}(\overline{\Omega})$, we have: 
\[\int_{\Omega}v\phi~dx=\lim_{n\rightarrow \infty}\int_{\Omega}v_{n}\phi~dx=\lim_{n\rightarrow \infty}-\int_{\Omega}z_{n}\cdot \nabla \phi~dx=-\int_{\Omega}z\cdot \nabla \phi~dx.\]
Since $-\int_{\Omega}z\cdot \nabla \phi~dx=\int_{\Omega}\Div(z)\phi~dx-\int_{\partial \Omega}[z,\nu]\phi~dH^{d-1}$, pick $\phi \in C^{\infty}_{0}(\Omega)$, we get $v=\Div(z)$. Consequently, $[z,\nu]=0$. So $v\in U_p$.
\end{proof}

\begin{lemma}
\label{conv:weak}
For $\{v_n\}_{n=1}^{\infty}\subset L^p_{\Box}(\Omega)$ and $v\in L^{p}_{\Box}(\Omega)$ with:
\[\sup_{n\in \mathbb{N}}\|v_n\|_{L_p}<\infty~~{\rm and}~~\lim_{n\rightarrow \infty}\|v_n-v\|_{Y_p}=0,\]
we have $v_n \rightharpoonup v$ weakly in $L_p(\Omega)$.
\end{lemma}
\begin{proof}
Since $\{v_n\}_{n=1}^{\infty}$ is bounded in $L_{p}(\Omega)$, we can extract a subsequence $\{v_{n_k}\}_{k=1}^{\infty}$ such that:
\begin{equation}
\label{conv:weak-1}
v_{n_k} \rightharpoonup \hat{v}~~~\mbox{weakly~in~} L_p(\Omega)
\end{equation}
for some $\hat{v}\in L^p_{\Box}(\Omega)$. Since $\lim_{n\rightarrow \infty}\|v_n-v\|_{Y_p}=0$, without loss of generality, we can assume $\sup_{n\in \N}\|v_n-v\|_{Y_p}<\infty$. By Theorem \ref{th:Y_p}, we can find $\{z_n\}\subset \dot{X}(\Omega)_p$ such that
\[v_n-v=\Div(z_n) {\rm ~~and~~}\|z_n\|_{L_\infty}=\|v_n-v\|_{Y_p}.\]

Since $[z_n, \nu]=0$, combining Theorem \ref{trace}, we have:
\begin{eqnarray*}
\int_{\Omega}\mbox{div}(z_{n_k})\phi~dx&=&-\int_{\Omega}z_{n_k}\cdot \nabla \phi~dx+\int_{\partial \Omega}[z_{n_k},\nu]tr(\phi)~dH^{d-1}\\
&=&-\int_{\Omega}z_{n_k}\cdot \nabla \phi~dx
\end{eqnarray*}
for all $\phi \in W^{1}(L_1(\Omega))\cap L^{q}(\Omega)$. Since $\lim_{k\rightarrow \infty}\|z_{n_k}\|_{L_\infty}=0$, we have:
\[\lim_{k\rightarrow \infty}\int_{\Omega}(v_{n_k}-v)\phi~dx=-\lim_{k\rightarrow \infty}\int_{\Omega}z_{n_k}\cdot \nabla \phi~dx=0.\]
This together with (\ref{conv:weak-1}) shows $v=\hat{v}$. Consequently, every subsequence of $\{v_n\}_{n=1}^{\infty}$ has in turn a weakly convergent subsequence with limit $v$. This implies that $v_n \rightharpoonup v$ weakly in $L_p(\Omega)$.
\end{proof}

With these preliminaries in hand, I am able to give two characterizations of the minimizing pair $(u_t,v_t)$.
The first is the following.
\begin{theorem}
\label{th:L2-BV}
Let $(u_{t},v_{t})$ be the minimizing pair of problem (\ref{equ:L2-BV-pair}) and $\bar{f}:=\frac{1}{|\Omega|}\int_{\Omega}f~dx$. We have:
\begin{description}
\item[(i)~] $\|f-\bar{f}\|_{Y_2} \leq t \Leftrightarrow u_{t}=\bar{f}$.
\item[(ii)] $\|f-\bar{f}\|_{Y_2} \geq t \Leftrightarrow \|v_{t}\|_{Y_2} = t$ and $\int_{\Omega} v_{t}u_{t}~dx=t|u_{t}|_{\BV}$.
\end{description}
\end{theorem}
\begin{proof}
$u_t$ is a minimizer for problem (\ref{equ:L2-BV}) is equivalent to say $(f-u_t)\in t\mathcal{A}(u_t)$. So there exists $z\in X(\Omega)_2$ with $\|z\|_{L_\infty}\leq 1$, $f-u_t=-t\Div(z)$ in $\mathcal{D'}(\Omega)$ such that $\int_{\Omega}(z,Du_t)=|u_t|_{\BV}$.
Since $\int_{\Omega}(z,Du_t)\leq \|z\|_{L_\infty}|u_t|_{\BV}\leq |u_t|_{\BV}$, the equality holds when $|u_t|_{\BV}=0$ or $\|z\|_{L_\infty}=1$.
\begin{enumerate}
\item When $|u_t|_{\BV}=0$, we have $u={\rm constant}$. To minimize $\|f-u_t\|_{L_2}$, we get $u_t=\frac{1}{|\Omega|}\int_{\Omega}f~dx$. So $\|f-\bar{f}\|_{Y_2}=\|-t\Div(z)\|_{Y_2}\leq t\|z\|_{L_\infty}\leq t$.
\item When $|u_t|_{\BV}>0$, we have $\|z\|_{L_\infty}=1$. So we have 
\[\int_{\Omega}v_tu_t~dx=-t\int_{\Omega}\Div(z)u_t~dx=t\int_{\Omega}(z,Du_t)=t|u_t|_{\BV}.\] 
And we claim $\|v_t\|_{Y_2}= t\|z\|_{L_\infty}= t$. Otherwise, there exists $\hat{z}\in \dot{X}(\Omega)_2$ with $v_t=-t\Div(\hat{z})$ and $\|\hat{z}\|_{L_\infty}< 1$, then $\int_{\Omega}v_tu_t~dx=t\int_{\Omega}(\hat{z},Du_t)<t|u_t|_{\BV}$, which contradicts our previous statement.
\end{enumerate}
\end{proof}
\begin{remark}
The above theorem is a generalization of Yves  Meyer's result in  \cite{Meyer} which was proved for the special case $\Omega=\R^{2}$ by using techniques from harmonic analysis.
\end{remark}

Antonin Chambolle  \cite{Chambolle} has introduced another type of characterization  for  a   finite dimensional  minimization  problem related to  (\ref{equ:L2-BV-pair}).  I have generalized his result  to the case $(L_2,\BV)$ as a consequence of theorem \ref{th:L2-BV}.

\begin{theorem}
\label{th:BV-proj}
Given $f\in L_2(\Omega)$, the minimizing pair for problem (\ref{equ:L2-BV-pair}) is $(u_{t},v_{t})=(f-\pi_{t U_{2}}(f),\pi_{tU_{2}}(f))$, where $\pi_{t U_{2}}(f)$ is the $L_2$ projection of $f$ onto the set $t U_{2}$.
\end{theorem}
\begin{proof}
~\\
1. When $\|f-\bar{f}\|_{Y_2} \leq t$, we have $v_t=\pi_{t U_2}(f)=f-\bar{f}\in t U_2$.\\
2. When $\|f-\bar{f}\|_{Y_2} \geq t$, by Theorem \ref{th:L2-BV}, we have $\|v_t\|_{Y_2}= t$
and $\int_{\Omega}v_tu_t~dx=t|u_t|_{\BV}$. For any $w\in tU_2$, there exists $z\in \dot{X}(\Omega)_2$ such that $w=-\mbox{div}(z)$ and $\|z\|_{L_\infty}=\|w\|_{Y_2}\leq t$. Thus
\[\int_{\Omega}w(f-v_t)~dx=\int_{\Omega}wu_t~dx= \int_{\Omega} -\mbox{div}(z)u_t~dx=\int_{\Omega}(z,Du_t).\]
So 
\[\int_{\Omega}w(f-v_t)~dx\leq \|z\|_{L_\infty}|u_t|_{\BV}\leq t |u_t|_{\BV}.\]
Thus we have:
\[\int_{\Omega}(w-v_t)(f-v_t)~dx\leq t |u_t|_{\BV}-t|u_t|_{\BV}=0\mbox{~for any~}w\in tU_2.\]
So we have $v_t=\pi_{tU_2}(f)$ and $u_t=f-\pi_{tU_2}(f)$.
\end{proof}

Since we have characterized the minimizing pair $(u_t,v_t)$, we can now give an alternative expression for the K-functional:
\[K(f,t):= \inf_{f=u+v} \{\frac{1}{2}\|v\|^{2}_{L_2}+t|u|_{\BV}\}.\]
\begin{theorem}
\label{K-characterization}
\[K(f,t)=\int_{\Omega}\pi_{t U_{2}}(f)f~dx-\frac{1}{2}\|\pi_{t U_{2}}(f)\|^2_{L_2}.\]
\end{theorem}
\begin{proof}
From Theorem \ref{th:L2-BV}, we know that $(u_t,v_t)=(f-\pi_{tU_2}(f),\pi_{tU_2}(f))$ and $\int_{\Omega}v_{t}u_{t}~dx=t |u_t|_{\BV}$. So
\[K(f,t)=\frac{1}{2}\|\pi_{t U_2}(f)\|^2_{L_2}+\int_{\Omega}(f-\pi_{t U_2}(f))\pi_{t U_2}(f)~dx=\int_{\Omega}\pi_{t U_2}(f)f~dx-\frac{1}{2}\|\pi_{t U_2}(f)\|^2_{L_2}.\]
\end{proof}

By using Theorem \ref{th:char}, we can calculate minimizers explicitly for some simple cases.
\begin{example}
Let $\Omega=B(0,R)=\{x\in \mathbb{R}^d:|x|\leq R\}$ and $f=\Chi_{B(0,r)}$ with $0<r<R$. When $0\leq t(|\partial B(0,r)|/|B(0,r)|+|\partial B(0,r)|/(|B(0,R)|-|B(0,r)|)) \leq 1$, we have:
\begin{eqnarray*}
u_t&=&(1-t\frac{d}{ r})\Chi_{B(0,r)}+t\frac{d\cdot r^{d-1}}{R^d-r^d} \Chi_{B(0,R)\setminus B(0,r)}\\
&=&(1-t\frac{|\partial B(0,r)|}{|B(0,r)|})\Chi_{B(0,r)}+t\frac{|\partial B(0,r)|}{|B(0,R)|-|B(0,r)|}\Chi_{B(0,R)\setminus B(0,r)}
\end{eqnarray*}
and
\begin{eqnarray*}
v_t&=&t\frac{d}{ r}\Chi_{B(0,r)}-t\frac{d\cdot r^{d-1}}{R^d-r^d} \Chi_{B(0,R)\setminus B(0,r)}\\
&=&t \frac{|\partial B(0,r)|}{|B(0,r)|}\Chi_{B(0,r)}-t \frac{|\partial B(0,r)|}{|B(0,R)|-|B(0,r)|}\Chi_{B(0,R)\setminus B(0,r)}.
\end{eqnarray*}
When $t(|\partial B(0,r)|/|B(0,r)|+|\partial B(0,r)|/(|B(0,R)|-|B(0,r)|))>1$, we have:
\begin{eqnarray*}
u_t&=&\frac{r^d}{R^d}\Chi_{B(0,R)}\\
&=&\frac{|B(0,r)|}{|B(0,R)|}\Chi_{B(0,R)}
\end{eqnarray*}
and
\begin{eqnarray*}v_t&=&\Chi_{B(0,r)}-\frac{r^d}{R^d}\Chi_{B(0,R)}\\
&=&\Chi_{B(0,r)}-\frac{|B(0,r)|}{|B(0,R)|}\Chi_{B(0,R)}.
\end{eqnarray*}
\end{example}
\begin{proof}
We look for the minimizer $u_t$ with the form $u_t=a \Chi_{B(0,r)}+b \Chi_{B(0,R)\setminus B(0,r)}$. Then $t \mbox{div}(z)=-(f-u_t)=(a-1) \Chi_{B(0,r)}+b \Chi_{B(0,R)\setminus B(0,r)}$.
We take $z=\frac{(a-1)}{td}x$ for $x\in B(0,r)$, then $u_t=f+t \mbox{div}(z)=a$ for $x\in B(0,r)$. To construct $z$ in $B(0,R)\setminus B(0,r)$, we will look for $z$ with the form
$z=\rho(|x|)\frac{x}{|x|}$. Since $\|z\|_{L_\infty}\leq 1$, we need $\rho(r)=-1$, this tells us $\frac{(a-1)}{td}=-\frac{1}{r}$. So $a=1-\frac{td}{r}$. To make $[z,\nu]=0$, we require $\rho(R)=0$. Since
\[\mbox{div}(z)=\nabla \rho(|x|)\cdot \frac{x}{|x|}+\rho(|x|)\mbox{div}(\frac{x}{|x|})=\rho'(|x|)+\rho(|x|)\frac{d-1}{|x|},\]
we must have:
\[\left \{\begin{array}{l}t(\rho'(s)+\rho(s)\frac{d-1}{s})=b~\mbox{for}~r<s<R\\ \rho(r)=-1\\ \rho(R)=0
\end{array}\right.\]
Solve this ODE, we get: $\rho(s)=-\frac{R^d r^{d-1}}{R^d-r^d}s^{1-d}+\frac{r^{d-1}}{R^{d}-r^{d}}s$ and $b=t\frac{d\cdot r^{d-1}}{R^d-r^d}$. In addition, $\rho'(s)=((d-1)(\frac{R}{s})^d+1)\frac{r^{d-1}}{R^d-r^d}\geq 0$, so $-1\leq \rho(s)\leq 0$ for $r<s<R$.

Thus, $u_t=(1-t\frac{d}{r})\Chi_{B(0,r)}+t\frac{d\cdot r^{d-1}}{R^d-r^d} \Chi_{B(0,R)\setminus B(0,r)}$ and $v_t=t\frac{d}{ r}\Chi_{B(0,r)}-t\frac{ d\cdot r^{d-1}}{R^d-r^d} \Chi_{B(0,R)\setminus B(0,r)}$.

To show $u_t$ is a minimizer for (\ref{equ:L2-BV}), we only need to check whether $\int_{\Omega}(z,Du_t)=|u_t|_{\BV}$. By Green's formula, we have:\begin{eqnarray*}\int_{\Omega}(z,Du_t)&=&-\int_{\Omega}\mbox{div}(z)u_t~dx=\frac{1}{t} \int_{\Omega}v_t u_t~dx\\
&=&\frac{1}{t}\{\int_{B(0,r)}(1-t\frac{d}{r})t\frac{d}{r}~dx+\int_{B(0,R)\setminus B(0,r)}t\frac{d\cdot r^{d-1}}{R^d-r^d}(-t\frac{d\cdot r^{d-1}}{R^d-r^d})~dx\}\\
&=&(1-t\frac{d}{r})\frac{d}{r}|B(0,r)|-t\frac{d\cdot r^{d-1}}{R^d-r^d}\frac{d\cdot r^{d-1}}{R^d-r^d}(|B(0,R)|-|B(0,r)|)\\
&=&(1-t\frac{d}{ r}-t\frac{d\cdot r^{d-1}}{R^d-r^d})H^{d-1}(\partial B(0,r))\\
&=&|u_t|_{\BV}.
\end{eqnarray*}
The above equality makes sense only when $1-t(\frac{d}{r}+\frac{d\cdot r^{d-1}}{R^d-r^d})\geq 0$, which means 
\[0\leq t(|\partial B(0,r)|/|B(0,r)|+|\partial B(0,r)|/(|B(0,R)|-|B(0,r)|))\leq 1.\]
\end{proof}

\begin{example}
Let $f(x)=x$ for $x\in [0,1]$.\\
For $0\leq t \leq \frac{1}{8}$, we have:
\[u_{t}=\left \{\begin{array}{ll}\sqrt{2t}&0\leq x\leq \sqrt{2t}\\x& \sqrt{2t}\leq x\leq 1-\sqrt{2t}\\1-\sqrt{2t}&1-\sqrt{2t}\leq x \leq 1\end{array}\right.\]
For $t > \frac{1}{8}$, we have: $u_{t}=\frac{1}{2}$ for $x\in [0,1]$.
\end{example}
\begin{proof}
Since $tz'=-v_t$ and $\|z\|_{L_\infty}\leq 1$, we look for $z$ with the following structure:
\[z=\left \{\begin{array}{ll}-\frac{1}{2t}(x-h)^2+1&0\leq x\leq h\\
1&h\leq x\leq 1-h\\
-\frac{1}{2t}(x-(1-h))^2+1&1-h\leq x \leq 1\end{array}\right.\]
We need to choose $h$ such that $z(0)=z(1)=0$(Neumann Boundary Condition). So we have $1-\frac{h^2}{2t}=0$. Hence $h=\sqrt{2t}$. But we also require $h\leq \frac{1}{2}$, so $\sqrt{2t}\leq \frac{1}{2}$.\\
Thus for $t \leq \frac{1}{8}$, we have:
\[z=\left \{\begin{array}{ll}-\frac{1}{2t}(x-\sqrt{2t})^2+1&0\leq x\leq \sqrt{2t}\\
1&\sqrt{2t}\leq x\leq 1-\sqrt{2t}\\
-\frac{1}{2t}(x-(1-\sqrt{2t}))^2+1&1-\sqrt{2t}\leq x \leq 1\end{array}\right.\]
Then we have:
\[u_{t}=\left \{\begin{array}{ll}\sqrt{2t}&0\leq x\leq \sqrt{2t}\\x& \sqrt{2t}\leq x\leq 1-\sqrt{2t}\\1-\sqrt{2t}&1-\sqrt{2t}\leq x \leq 1\end{array}\right.\]
For $t > \frac{1}{8}$, we have $z=-4(x-\frac{1}{2})^2+1$ for $x\in [0,1]$. Then we have: $u_{t}=\frac{1}{2}$ for $x\in [0,1]$.
To show $u_t$ is a minimizer, we only need to check $\int_{\Omega}(z,Du_{t})=|u_{t}|_{\BV}$:
\[\int_{0}^{1}(z,Du_{t})=\int_{h}^{1-h}z~dx=1-2h=|u_{t}|_{\BV}.\]
\end{proof}

\section{Multiscale Decompositions}
The solution of minimization problems like  (\ref{equ:L2-BV-pair}) leads to multiscale decompositions of a general function $f$.  In the case we have been considering, each $f\in L_2(\Omega)$ is decomposed as $f=\sum_{k=0}^\infty w_k$ where each $w_k$ is viewed as providing the detail of $f$ at some scale.
Currently, there are several ways to achieve this goal.    The most common of these is to use a standard telescoping  decomposition where $w_k:= u_{t_k}-u_{t_{k-1}}$ and $t_k=2^{-k}$.   Other  approaches to obtain multiscale decompositions were given by  Eitan Tadmor et al.'s work (see \cite{Tadmor})  and Stanely Osher et al.'s work (see \cite{Osher2}).

Since $\|v_t\|_{Y_2}$ for problem (\ref{equ:L2-BV-pair}) depends on  the parameter $t$, this gives  us a way of decomposing a given function $f\in L_2(\Omega)$ into different components based on the size of the $\|\cdot\|_{Y_2}$ norm of each component.
This approach falls into a category of methods (called {\it Inverse Scale Space Methods}) that  were introduced by Groetsch and Scherzer in \cite{Groetsch}.  It centers on using the above $\|\cdot\|_{Y_p}$ norm to measure the oscillation of $v_t$ in a cetain sense.
In our language, the choice of components takes  the following form:
\begin{equation}
\label{MsD:1}
u_{k+1}:=\argmin~\{\frac{1}{2}\|f-u\|^2_{L_2}+t_{k}J(u,u_k)\},
\end{equation}
where $\frac{1}{2}\|f-u\|^2_{L_2}$ is the $L_2$-norm fit-to-data term and $J(u,u_k)$ is a regularization term. Typically we initialize $u_0=0$ or $u_0=\frac{1}{|\Omega|}\int_{\Omega}f~dx$ and we require that $\{u_k\}$ satisfies the {\it inverse fidelity} property:
\[\lim_{k\rightarrow \infty}\|f-u_k\|_{L_2}\rightarrow 0.\]

If, as a special case, we consider $\BV$ minimization and choose $J(u,u_k)$ as the Bregman distance  defined by
\[D(u,u_k):=|u|_{\BV}-|u_k|_{\BV}-\int_{\Omega}s(u-u_k)~dx,\]
where $s\in \partial(|u_k|_{\BV})$, then (\ref{MsD:1}) becomes the method  introduced by Osher et.al. in \cite{Osher2}. This method has many promising properties for image denoising which were proved in \cite{Osher2}.

The interpretation of the Bregman distance for image processing given above is somewhat ambiguous.
Another  possibility is to simply take $J(u,u_k):=|u-u_k|_{\BV}$. Roughly speaking, $|u_{k+1}-u_k|_{\BV}$ measures the similarity between two images $u_k$ and $u_{k+1}$.  For any choice  $t_{k}>0$,    $u_{k+1}$   contains more detail  than $u_k$ and is  closer to $f$. We choose a sequence  $t_0>t_1>t_2>...$ with $\lim_{n \rightarrow \infty}t_n=0$. One sees that
\begin{equation}
\label{MsD:2}
u_{k+1}:=\argmin_{u\in \BV(\Omega)\cap L_2(\Omega)}~\{\frac{1}{2}\|f-u\|^2_{L_2}+t_{k}|u-u_k|_{\BV}\},
\end{equation}
with $u_0=0$. If we take $w_{k+1}:=u_{k+1}-u_k$, $w_0:=u_0$ and $v_k:=f-u_k$, then (\ref{MsD:2}) can be viewed as:
\[(w_{k+1},v_{k+1})=\argmin_{w+v=v_k}~\{\frac{1}{2}\|v\|^2_{L_2}+t_{k}|w|_{\BV}\},\]
with $v_0=f$. Thus $u_{k+1}=\sum_{n=1}^{k+1}w_n$ will be a minimizer for (\ref{MsD:2}). This is  the hierarchical $(L_2,\BV)$ decomposition method   introduced by Tadmor et.al. in \cite{Tadmor}.
\begin{theorem}
\label{Conv:1}
Let $f\in L_2(\Omega)$ and $\{u_k\}_{k=1}^{\infty}$ be defined as in (\ref{MsD:2}). Then we have:
\begin{enumerate}
\item $\|f-u_{k+1}\|_{L_2}\leq \|f-u_k\|_{L_2}.$
\item $\|f-u_{k+1}\|_{Y_2}\leq t_k \rightarrow 0.$
\item Let $v_n:=f-u_n$, then $\sum_{k=0}^{n}\{2t_k |u_{k+1}-u_k|_{\BV}+\|u_{k+1}-u_k\|^2_{L_2}\}=\|f\|^2_{L_2}-\|v_{n+1}\|^2_{L_2}$.
\end{enumerate}
\end{theorem}
\begin{proof}~\\
1. Since
\[\frac{1}{2}\|f-u_k\|^2_{L_2}\geq \frac{1}{2}\|f-u_{k+1}\|^2_{L_2}+t_k|u_{k+1}-u_k|_{\BV} \geq \frac{1}{2}\|f-u_{k+1}\|^2_{L_2},\]
we have
\[\|f-u_{k+1}\|_{L_2}\leq \|f-u_k\|_{L_2}.\]\\
2. By Theorem \ref{th:char}.\\
3. If $\|v_k-c_k\|_{Y_2}< t_k$, where $c_k:=\frac{1}{|\Omega|}\int_{\Omega}v_k~dx$, then $(u_{k+1}-u_k,v_{k+1})=(c_k,v_k-c_k)$. Otherwise, we have:
\begin{equation}
\label{eqn:1}
\|v_{k+1}\|_{Y_2}= t_k,~~~\langle v_{k+1}, u_{k+1}-u_{k} \rangle:=\int_{\Omega}v_{k+1}(u_{k+1}-u_{k})~dx=t_k|u_{k+1}-u_{k}|_{\BV}.
\end{equation}
Since $u_{k+1}-u_k+v_{k+1}=v_k$, we get:
\begin{equation}
\label{eqn:2}
\|v_k\|^2_{L_2}=\langle (u_{k+1}-u_k)+v_{k+1}, (u_{k+1}-u_k)+v_{k+1} \rangle = \|v_{k+1}\|^2_{L_2}+\|u_{k+1}-u_k\|^2_{L_2}+2\langle v_{k+1}, u_{k+1}-u_{k} \rangle.
\end{equation}
Combining (\ref{eqn:1}) and (\ref{eqn:2}), we get:
\begin{equation}
\label{eqn:3}
\|v_k\|^2_{L_2}-\|v_{k+1}\|^2_{L_2}=\|u_{k+1}-u_k\|^2_{L_2}+2t_k|u_{k+1}-u_k|_{\BV}.
\end{equation}
Sum (\ref{eqn:3}) from $k=0$ to $k=n$, we get:
\begin{equation}
\label{eqn:4}
\sum_{k=0}^{n}\{2t_k |u_{k+1}-u_k|_{\BV}+\|u_{k+1}-u_k\|^2_{L_2}\}=\|f\|^2_{L_2}-\|v_{n+1}\|^2_{L_2}.
\end{equation}
\end{proof}
In addition, we have the following $L_2$ convergence result:
\begin{theorem}
Let $f\in L_2(\Omega)$ with $\Omega$ bounded Lipschitz domain in $\R^d$ and $t_k=t_{0}\cdot r^k$ with $0<r<1$, then we have:
\begin{enumerate}
\item $\lim_{k\rightarrow \infty}\|f-u_k\|_{L_2}=0.$
\item $\sum_{k=0}^{\infty}\{2t_k |u_{k+1}-u_k|_{\BV}+\|u_{k+1}-u_k\|^2_{L_2}\}=\|f\|^2_{L_2}$.
\end{enumerate}
\end{theorem}
\begin{proof}~\\
1. By Theorem \ref{Conv:1}, we know that $\{\|v_n\|_{L_2}\}$ is a decreasing sequence. Hence, to prove $\|v_n\|_{L_2}\rightarrow 0$, we only need to show $\|v_{2n+1}\|_{L_2}\rightarrow 0$. Note $v_{2n+1}=v_n-\sum_{k=n}^{2n}(u_{k+1}-u_k)$. Multiply $v_{2n+1}$ with itself, we get:
\[\|v_{2n+1}\|^2_{L_2}=-\langle v_{2n+1},\sum_{k=n}^{2n}(u_{k+1}-u_k) \rangle+\langle v_{2n+1},v_n \rangle=: A+B.\]
By Theorem \ref{Conv:1}, we know $\|v_{2n+1}\|_{Y_2}\leq t_{2n}$. So
\[|A|\leq t_{2n}|\sum_{k=n}^{2n}(u_{k+1}-u_k)|_{\BV}\leq t_{2n}\sum_{k=n}^{2n}|u_{k+1}-u_k|_{\BV}\leq \sum_{k=n}^{2n}t_k|u_{k+1}-u_k|_{\BV}.\]
From Theorem \ref{Conv:1}, we know that $\sum_{k=0}^{n}t_k |u_{k+1}-u_k|_{\BV}\leq \frac{1}{2}\|f\|^2_{L_2}$. Hence $\{\sum_{k=0}^{n}t_k |u_{k+1}-u_k|_{\BV}\}$ is a Cauchy sequence. So $|A|\rightarrow 0$ when $n\rightarrow \infty$.\\
Since $v_n=f-u_n=f-\sum_{k=0}^{n-1}(u_{k+1}-u_k)$, we have:
\begin{eqnarray*}|B|&=&|\langle v_{2n+1},f \rangle-\sum_{k=0}^{n-1}\langle v_{2n+1},u_{k+1}-u_k\rangle|\\
&\leq& |\langle v_{2n+1},f \rangle|+t_{2n}\sum_{k=0}^{n-1}|u_{k+1}-u_k|_{\BV}\\
&\leq& |\langle v_{2n+1},f \rangle|+\frac{t_{2n}}{t_n}\sum_{k=0}^{n-1}t_k|u_{k+1}-u_k|_{\BV}\\
&\leq& |\langle v_{2n+1},f \rangle|+\frac{t_{2n}}{2t_n}\|f\|^2_{L_2}.
\end{eqnarray*}
Since $\lim_{n\rightarrow 0}\frac{t_{2n}}{2t_n}=0$, we have $|B|\rightarrow 0$ iff $|\langle v_{2n},f \rangle|\rightarrow 0$.\\
Since $u_{k+1}$ is a minimizer for (\ref{MsD:2}), there exists $z_{k+1}\in X(\Omega)_2$ with $\|z_{k+1}\|_{L_\infty}\leq 1$, $v_{k+1}=f-u_{k+1}=-t_{k}\Div(z_{k+1})=\Div(-t_{k}z_{k+1})\in L^2_{\Box}(\Omega)$ such that
\[\left \{\begin{array}{ll}&\int_\Omega (z_{k+1},D(u_{k+1}-u_k))=|u_{k+1}-u_k|_{\BV} \\&[z_{k+1},\nu]=0\end{array}\right.\]
Since we have $\sup_{k\in \mathbb{N}}\|v_k\|_{L_2}\leq \|f\|_{L_2}$ and $\|v_{k+1}\|_{Y_2}\leq t_k\rightarrow 0$, by Theorem \ref{conv:weak}, we know $v_n\rightharpoonup 0$ weakly in $L_2(\Omega)$. So we have $|\langle v_{2n+1},f \rangle|\rightarrow 0$. Consequently $\|v_{2n+1}\|_{L_2}\rightarrow 0$. So $\lim_{k\rightarrow \infty}\|f-u_k\|_{L_2}=0$.

2. Recall from Theorem \ref{MsD:2}, we have:
\[\sum_{k=0}^{n}\{2t_k |u_{k+1}-u_k|_{\BV}+\|u_{k+1}-u_k\|^2_{L_2}\}=\|f\|^2_{L_2}-\|v_{n+1}\|^2_{L_2}.\]
Let $n\rightarrow \infty$ and notice $\|v_n\|_{L_2}\rightarrow 0$, we have:
\[\sum_{k=0}^{\infty}\{2t_k |u_{k+1}-u_k|_{\BV}+\|u_{k+1}-u_k\|^2_{L_2}\}=\|f\|^2_{L_2}\]
\end{proof}
\begin{remark}In \cite{Tadmor} (Tadmor et al.), the same result was proved under the assumption $f\in \BV(\Omega)$($\Omega\subset \R^2$) or $f \in (L_2,\BV)_{\theta}$ with $0<\theta<1$. Here we have removed the smoothness assumption.
\end{remark}

\section{Decomposition for $(X,Y)=(L_p,W^{1}(L_\tau))$}
Now let's consider the following $(L_p(\Omega),W^{1}(L_\tau(\Omega)))$ decomposition with $1/\tau:=1/p+1/d$:
\begin{equation}
\label{Lp-Sobolev}
(u_t,v_t):=\argmin_{u+v=f}~\{\frac{1}{p}\|v\|^p_{L_p}+t \|\nabla u\|_{L_\tau}\},
\end{equation}
where $\|\nabla u\|_{L_\tau}:=(\int_{\Omega}|\nabla u|^{\tau}~dx)^{1/\tau}$ and $1<p<\infty$. In this section, $\Omega:=\R^d$ or $\Omega$ is a bounded Lipschitz domain in $\R^d$. Thus by Sobolev embedding theorem, we have $W^{1}(L_\tau(\Omega))\subset L_p(\Omega)$.
Let
\[J_{\tau}(u):=\left \{ \begin{array}{ll}\|\nabla u\|_{L_\tau}&u\in W^{1}(L_{\tau}(\Omega))\\
+\infty&u\in L_p(\Omega) \setminus  W^{1}(L_{\tau}(\Omega))\end{array} \right.\]
be a functional defined on $L_p(\Omega)$. Since $(L_p)^*=L_q$ for $\frac{1}{p}+\frac{1}{q}=1$, by Riesz representation theorem, any functional $s$ defined on $L_p(\Omega)$ can be represented as $\langle s,v \rangle:=\int_{\Omega}sv~dx$ for any $v\in L_p(\Omega)$.
To study the pair $(L_p(\Omega),W^{1}(L_\tau(\Omega)))$, we first need to characterize the subdifferential of the functional $J_{\tau}$. For $s\in L^{q}_{\Box}(\Omega)$ $(\frac{1}{p}+\frac{1}{q}=1)$, we can define the norm $\|\cdot\|_{G_\tau}$ by
\[\|s\|_{G_\tau}:=\sup_{\|\nabla v\|_{L_\tau}\leq 1}\int_{\Omega}sv~dx\]

\begin{theorem}~
\begin{description}
\item[(i)~]For $u\neq {\rm constant}$, $\partial J_{\tau}(u):=\{s \in L^{q}_{\Box}(\Omega)~:~\int_{\Omega}su~dx=\|\nabla u\|_{L_\tau}~\mbox{and}~\|s\|_{G_\tau} = 1\}$.
\item[(ii)]For $u= {\rm constant}$, $\partial J_{\tau}(u):=\{s \in L^{q}_{\Box}(\Omega)~:~\|s\|_{G_\tau}\leq 1\}$.
\end{description}
\end{theorem}
\begin{proof}
Given a function $u \in W^{1}(L_\tau(\Omega))$, we can define a functional $\hat{s}$ on $\mathbb{R} u$ such that $\langle \hat{s},c u \rangle =c\|\nabla u\|_{L_\tau}$ for any $c\in \R$. By Hahn-Banach Theorem, we can extend the domain of the  functional to $W^{1}(L_\tau(\Omega))$. Let's say functional $s$ with $s|_{\mathbb{R} u}=\hat{s}$ and $|\langle s,v \rangle| \leq \|\nabla v\|_{L_\tau}$ for any $v \in W^{1}(L_\tau(\Omega))$. Since $s|_{\mathbb{R} u}=\hat{s}$, we have $\langle s,u \rangle=\|\nabla u\|_{L_\tau}$. Consequently, we have $\|s\|_{G_\tau}\leq 1$ for $s \in L^{q}_{\Box}(\Omega)$(equality holds when $\nabla u \neq 0$).

Then if $s \in L^{q}_{\Box}(\Omega)$, we have: 
\[\langle s,v-u \rangle=\int_{\Omega}s(v-u)~dx\leq \|s\|_{G_\tau}\|\nabla v\|_{L_\tau}-\|\nabla u\|_{L_\tau}\leq J_{\tau}(v)-J_{\tau}(u)\]
for any $v\in W^{1}(L_\tau(\Omega))$.
So $s \in \partial J_{\tau}(u)$.

Conversely, if $s\in \partial J_{\tau}(u)$, then $J_{\tau}(v)-J_{\tau}(u) \geq \langle s,v-u \rangle$ for any $v\in W^{1}(L_\tau(\Omega))$. In addition, we have $s\in L_q(\Omega)$. By taking $v=\lambda u$, we get
\[(1-\lambda)(\langle s,u \rangle -\|\nabla u\|_{L_\tau})\geq 0.\]
By successively taking $\lambda >1$ and $\lambda < 1$, we deduce that $\langle s,u \rangle=\|\nabla u\|_{L_\tau}$. Therefore, $\langle s, v\rangle\leq \|\nabla v\|_{L_\tau}$ for all $v\in  W^{1}(L_\tau(\Omega))$ and $\langle s,u\rangle = \|\nabla u\|_{L_\tau}$. This implies that $\|s\|_{G_\tau}\leq 1$(equality holds when $\nabla u \neq 0$). Since $\langle s,c \rangle\leq J_\tau(u+c)-J_\tau(u) =0$ for any constant $c$, we have $\int_{\Omega}s~dx=0$, which means $s\in L^q_{\Box}(\Omega)$.
\end{proof}

\noindent Define the duality mapping $\mathfrak{J}_p:L_p(\Omega)\mapsto L_q(\Omega)$ $(\frac{1}{p}+\frac{1}{q}=1)$ by:
\[\mathfrak{J}_p(u):=|u|^{p-2}u.\]
It is easy to check $\mathfrak{J}_p(u)=\partial(\frac{1}{p}\|\cdot\|^{p}_{L_p})(u)$. Then we have the following theorem for the minimizing pair $(L_p(\Omega),W^{1}(L_\tau(\Omega)))$.
\begin{theorem}
\label{th:Lp-Sobolev}
Given $f\in L_p(\Omega)$ and let $(u_{t},v_{t})$ be the minimizing pair of problem (\ref{Lp-Sobolev}) and $c_f:=\argmin_{c\in \R}\|f-c\|_{L_p}$. We have:
\begin{enumerate}
\item $\|\mathfrak{J}_p(f-c_f)\|_{G_\tau} \leq t \Leftrightarrow u_{t}=c_f$.
\item $\|\mathfrak{J}_p(f-c_f)\|_{G_\tau} \geq t \Leftrightarrow \|\mathfrak{J}_p(v_{t})\|_{G_\tau} = t$ and $\int_{\Omega} \mathfrak{J}_p(v_{t})u_{t}~dx=t\|\nabla u_{t}\|_{L_\tau}$.
\end{enumerate}
\end{theorem}
\begin{proof}
Since $\langle \mathfrak{J}_p(f-c_f),c \rangle\leq \frac{1}{p}(\|f-c_f+c)\|^p_{L_p}-\|f-c_f\|^p_{L_p})$ for any $c\in \R$ and the right handside of the inequality is always nonnegative, we get $\langle \mathfrak{J}_p(f-c_f),c \rangle=0$ for any $c\in \R$, which means $\int_{\Omega}\mathfrak{J}_p(f-c_f)~dx=0$. So $\mathfrak{J}_{p}(f-c_f)\in L^{q}_{\Box}(\Omega)$.

$u_{t}$ is a minimizer for problem (\ref{Lp-Sobolev}) is equivalent to say $\mathfrak{J}_{p}(f-u_{t})\in t \partial J_{\tau}(u_{t})$. So $\mathfrak{J}_p(v_t)\in L^q_{\Box}(\Omega)$. In addition, we have $\|\mathfrak{J}_{p}(v_{t})\|_{G_\tau}\leq t$ and $\int_{\Omega}\mathfrak{J}_p(v_{t})u_{t}~dx=t\|\nabla u_{t}\|_{L_\tau}$.
\begin{enumerate}
\item When $\|\mathfrak{J}_p(f-c_f)\|_{G_\tau}\leq t$, we have $\mathfrak{J}_{p}(f-c_f)\in t \partial J_{\tau}(c_f)$. So $u_t=c_f$.
\item When $\|\mathfrak{J}_p(f-c_f)\|_{G_\tau}>t$, $u_t$ cannot be a constant. So $\|\mathfrak{J}_p(v_{t})\|_{G_\tau} = t$.
\end{enumerate}
\end{proof}

\section{Decomposition for $(X,Y)=(\ell_2,\ell_p)$}
A special case that is important in analysis and in numerical methods is when $X$ and $Y$ are a pair of $\ell_p$ spaces.   Such problems occur when we discretize the decomposition problems for Sobolev or Besov spaces and also when we develop numerical methods. In this chapter we shall study the minimizing pair for the case of $X=\ell_2:=\ell_2(\Z)$ and $Y=\ell_p:=\ell_p(\Z)$, $1\le p<\infty$, i.e. the problem

\begin{equation}
\label{l2-lp}
(x_t,y_t):=\argmin_{x+y=b}~\{\frac{1}{2}\|y\|^{2}_{\ell_2}+t\|x\|_{\ell_p}\},
\end{equation}
where $b\in \ell_2$.

\begin{theorem}~
\begin{enumerate}
\item For $x \neq 0$, \(\partial(\|x\|_{\ell_p}):=\{s \in \ell_{q}~:~s\cdot x=\|x\|_{\ell_p}~\mbox{and}~\|s\|_{\ell_{q}} = 1\}\).
\item For $x = 0$, \(\partial(\|x\|_{\ell_p}):=\{s \in  \ell_{q}~:~\|s\|_{\ell_q} \leq 1\}\).
\end{enumerate}
\end{theorem}
\begin{proof}
Given a sequence $x \in \ell_p$, we can define a functional $\hat{s}$ on $\mathbb{R} x$ such that $\langle \hat{s},c x \rangle =c\|x\|_{\ell_p}$ for any $c\in \R$. By Hahn-Banach Theorem, we can extend the domain of the  functional to $\ell_p$. Let's say functional $s$ with $s|_{\mathbb{R} x}=\hat{s}$ and $|\langle s,y \rangle| \leq \|y\|_{\ell_p}$ for any $y \in \ell_p$. Since $s|_{\mathbb{R} x}=\hat{s}$, we have $\langle s,x \rangle=\|x\|_{\ell_p}$. So $\|s\|_{\ell_q}\leq 1$ and we have $\|s\|_{\ell_q}=1$ for the case $x \neq 0$. Then we have 
\[s\cdot (y-x)\leq \|s\|_{\ell_q}\|y\|_{\ell_p}-\|x\|_{\ell_p}\leq \|y\|_{\ell_p}-\|x\|_{\ell_p}.\]
So $s \in \partial(\|x\|_{\ell_p})$.

Conversely, if $s \in \partial(\|x\|_{\ell_p})$, then $\|y\|_{\ell_p}-\|x\|_{\ell_p} \geq s\cdot (y-x)$. By taking $y=\lambda x$, we get
\[(1-\lambda)(s\cdot x -\|x\|_{\ell_p})\geq 0.\]
By successively taking $\lambda >1$ and $\lambda < 1$, we deduce that $s\cdot x=\|x\|_{\ell_p}$. Therefore, $s\cdot y\leq \|y\|_{\ell_p}$ for all $y \in \ell_p$ and $s\cdot x = \|x\|_{\ell_p}$. This implies that $\|s\|_{\ell_q}\leq 1$ (equality holds when $x \neq 0$).
\end{proof}

\noindent For the case $p=1$, the minimizing pair $(x_t,y_t)$ can be obtained by the \lq\lq soft thresholding\rq \rq~procedure which is widely used for wavelet shrinkage in image processing (See \cite{DeVore}). That is to say: $x^{i}_t=\mbox{sign}(b^i)\max\{0,|b^i|-t\}$, where $x_t=\{x^{i}_t\}$ and $b=\{b^{i}\}$. It can be shown that the \lq\lq soft thresholding\rq \rq~technique is a special case of the following characterization.

\begin{theorem}
\label{th:l2-lp}
Given $b\in \ell_2$, let $(x_{t},y_{t})$ be the minimizing pair of problem (\ref{l2-lp}). We have:
\begin{enumerate}
\item $\|b\|_{\ell_q} \leq t \Leftrightarrow x_{t}=0$.
\item $\|b\|_{\ell_q} \geq t \Leftrightarrow \|y_{t}\|_{\ell_q} = t$ and $y_{t}\cdot x_{t}=t\|x_t\|_{\ell_p}$.
\end{enumerate}
\end{theorem}
\begin{proof}
$(x_t,y_t)$ is the minimizing pair of problem (\ref{l2-lp}) is equivalent to say: $b-x_t\in t\partial(\|x_t\|_{\ell_p})$. Consequently, we have $y_t\cdot x_t=(b-x_t)\cdot x_t=t\|x_t\|_{l_p}$.
\begin{enumerate}
\item When $\|b\|_{\ell_q}\leq t$, we have $b\in t\partial(\|\cdot\|_{\ell_p})(0)$. So $x_t=0$.
\item When $\|b\|_{\ell_q}>t$, $x_t$ cannot be the zero element. So $\|y_t\|_{\ell_q}=\|b-x_t\|_{\ell_q}=t$.
\end{enumerate}
\end{proof}

\noindent Let's define the convex set 
\[ U_{q}:=\{y\in \ell_{2}:\|y\|_{\ell_q}\leq 1\}\]
 where $q$ is the dual index to $p$ ($1/q+1/p=1$).
Then the set $U_q$ is closed in $\ell_2$ norm topology, so the $\ell_2$ projection of a given sequence $b\in \ell_2$ onto the set $U_q$ is always well defined.
\begin{theorem}
\label{th:lp-proj}
Given $b\in \ell_2$, the minimizing pair for problem (\ref{l2-lp}) is $(x_{t},y_{t})=(b-\pi_{tU_q}(b),\pi_{tU_q}(b))$, where $\pi_{t U_q}(b)$ is the $\ell_2$ projection of $b$ onto the set $t U_q$.
\end{theorem}
\begin{proof}~\\
1. When $\|b\|_{\ell_q}\leq t$, we have $y_t=\pi_{t U_q}(b)=b\in t U_q$.\\
2. When $\|b\|_{\ell_q}\geq t$, by Theorem \ref{th:l2-lp}, we have $\|y_{t}\|_{\ell_q}=t$ and $y_{t}\cdot x_{t}=t\|x_t\|_{\ell_p}$. For any $z_t\in t U_q$, we have:
\[(z_t-y_t)\cdot (b-y_t)=z_t\cdot x_t-y_t\cdot x_t\leq \|z_t\|_{\ell_q}\|x_t\|_{\ell_p}-t\|x_t\|_{\ell_p}\leq 0.\]
So $y_t=\pi_{t U_q}(b)$. Consequently, $x_t=b-\pi_{t U_q}(b)$.
\end{proof}
\begin{remark}
This result paved a road for studying $(L_2, B^{\alpha}_{p}(L_{p}))$ type decompositions, where $B^{\alpha}_{p}(L_{p})$ is a Besov space. Since the norms of $L_2$ and Besov spaces can be characterized by their wavelet coefficients, the $(L_2,B^{\alpha}_{p}(L_{p}))$ type decompositions can be derived from the  decompositions for sequence spaces.
\end{remark}

\section{Numerical Implementation for $(X,Y)=(L_2(\Omega),\BV(\Omega))$}
The problem of finding a minimizing pair $(u_t,v_t)$ almost always is solved numerically.
Typically, numerical methods are built through some discretization of the continuous problem. In this chapter, we will study the numerical implementation for $(L_2(\Omega),\BV(\Omega))$ decomposition. Without loss of generality, we can assume $\Omega=[0,1]^d$. Let 
\[D_{n}:=\{2^{-n}([k_1,k_1+1]\times[k_2,k_2+1]\times \dotsi \times [k_d,k_d+1])~:~0\leq k_i \leq 2^n-1,~1\leq i \leq d\}\]
 be dyadic cubes with length $2^{-n}$.
We replace $f$ by a piecewise constant approximation of $f$ which has the following form:
\[f_n:=\sum_{I \in D_n}a_{I}\Chi_I.\]
Thus, $f_n$ is a function in the linear space
\[V_n:=\{f\in L_2(\Omega)~:~f\mbox{~is constant on~}I,~I \in D_n\}.\]

The spaces $\{V_n\}$, $n\ge 0$,  are nested, i.e., $V_{0}\subset \dotsi \subset V_{n}\subset V_{n+1} \subset \dotsi$. As our approximation, we will take  $f_n=P_{n}(f)$, where $P_{n}(f)$ is the $L_2$ projection of $f$ onto $V_n$.  The original problem (\ref{equ:L2-BV}) is then approximated by  the finite-dimensional problem
\begin{equation}
\label{num:1}
u_n:=\argmin_{u\in V_n}~\{\frac{1}{2}\|f_n-u\|^2_{L_2}+t|u|_{\BV}\}.
\end{equation}
We can compute $\|f_n-u_n\|_{L_2}$ and $|u_n|_{\BV}$ exactly by discrete norms. Problem (\ref{num:1}) then becomes a $\ell_1$-minimization problem:
\begin{equation}
\label{num:2}
x_{t}:=\argmin_{x}~\{\|b-x\|^{2}_{\ell_2}+t\|Mx\|_{\ell_1}\},
\end{equation}
where $M$ is an $m\times n$ matrix with the property $M(x+c)=Mx$ for any constant vector $c$.

Define 
\[K_{n}(f,t):=\inf_{u\in V_n} ~\{\frac{1}{2}\|f_n-u\|^2_{L_2}+t|u|_{\BV}\}.\] 

We have the following result about the convergence of the discrete solution to the continuous solution:
\begin{theorem}
Let $u_t$ be the minimizer of problem (\ref{equ:L2-BV}). We have:
\begin{enumerate}
\item $u_{n}\rightharpoonup u_t$ weakly in $L_2$.
\item $u_{n}\rightarrow u_t$ strongly in $L_{p}$, where $1\leq p < d/(d-1)$.
\item $\lim_{n\rightarrow \infty}K_{n}(f,t)=K(f,t)$.
\end{enumerate}
\end{theorem}
\begin{proof}
Since $\frac{1}{2}\|f_n-u_n\|^2_{L_2}\leq \frac{1}{2}\|f_n-u_n\|^2_{L_2}+t|u_n|_{\BV}\leq\frac{1}{2}\|f_n\|^2_{L_2}$, we have $\|u_n\|_{L_2}\leq \|f_n\|_{L_2}$. Similarly, we can get $|u_n|_{\BV}\leq \frac{1}{2t}\|f_n\|^2_{L_2}$. Since $f_n=P_n(f)\rightarrow f$ in $L_2$, $\{\|u_n\|_{L_2}\}$ and $\{|u_n|_{\BV}\}$ are bounded sequences. So we can find a subsequence $\{u_{n_k}\}$ such that:
\begin{enumerate}
\item $u_{n_k}\rightharpoonup \hat{u}$ weakly in $L_2$.
\item $u_{n_k}\rightarrow \hat{u}$ strongly in $L_{p}$ for $1\leq p<d/(d-1)$. (by the compact embedding $\BV(\Omega)\hookrightarrow L_p(\Omega)$)
\item $|u_{n_k}|_{\BV}\rightarrow |\hat{u}|_{\BV}$.
\end{enumerate}
for some $\hat{u}\in \BV(\Omega)\cap L_2(\Omega)$.
Then we have:
\begin{equation}
\label{num:3}
K(f,t)=\frac{1}{2}\|f-u_t\|^2_{L_2}+t|u_t|_{\BV}\leq \frac{1}{2}\|f-\hat{u}\|^2_{L_2}+t|\hat{u}|_{\BV}\leq \liminf_{k\rightarrow \infty}K_{n_k}(f,t).
\end{equation}
Given the minimizer $u_t\in \BV(\Omega)\cap L_2(\Omega)$, we can find a sequence $\{w_n\}$ with $w_n\in V_n$ such that:
\[w_n\rightarrow u_t~{\rm in}~L_2~~~{\rm and}~~~|w_n|_{\BV}\rightarrow |u_t|_{\BV}.\]
Since $K_n(f,t)=\frac{1}{2}\|f_n-u_n\|^2_{L_2}+|u_n|_{\BV}\leq \frac{1}{2}\|f_n-w_n\|^2_{L_2}+|w_n|_{\BV}$, we have:
\begin{equation}
\label{num:4}
\limsup_{n\rightarrow \infty}K_n(f,t)\leq \limsup_{n\rightarrow \infty}~\{\frac{1}{2}\|f_n-w_n\|^2_{L_2}+|w_n|_{\BV}\}=K(f,t).
\end{equation}
Combining (\ref{num:3}) and (\ref{num:4}), we know that $\hat{u}$ is a minimizer for problem (\ref{equ:L2-BV}) and 
\begin{equation}
\label{num:5}
\limsup_{n\rightarrow \infty}K_n(f,t)=\liminf_{k\rightarrow \infty}K_{n_k}(f,t)=K(f,t).
\end{equation}
By the uniqueness of the minimizer, we have $\hat{u}=u_t$. Since we can extract a further subsequence $\{u_{n_k}\}$, which satisfies (\ref{num:5}) and has the limit $u_t$, from any subsequence of $\{u_n\}$, we have:
\begin{enumerate}
\item $u_{n}\rightharpoonup u_t$ weakly in $L_2$.
\item $u_{n}\rightarrow u_t$ strongly in $L_{p}$, where $1\leq p < d/(d-1)$.
\item $\lim_{n\rightarrow \infty}K_{n}(f,t)=K(f,t)$.
\end{enumerate}
\end{proof}

\end{document}